\documentclass[11pt, a4paper]{article}

\usepackage[ngerman,english]{babel} 
\usepackage[iso]{umlaute}
\usepackage{a4wide}
\usepackage{amsmath}
\usepackage{amssymb}
\usepackage{amsthm}
\usepackage{mathabx}
\usepackage{bm}
\usepackage{graphicx}
\usepackage{caption}
\usepackage{subcaption}
\usepackage[margin=1in]{geometry}
\usepackage[latin1]{inputenc}
\usepackage[colorlinks=true,linkcolor=blue,urlcolor=red,citecolor=blue]{hyperref}
\usepackage{enumerate}
\usepackage{ifthen}
\usepackage{verbatim}
\usepackage{booktabs}
\usepackage{color}
\usepackage{float}
\usepackage{multicol}

\definecolor{seq1}{RGB}{12,44,132}
\definecolor{seq2}{RGB}{34,94,168}
\definecolor{seq3}{RGB}{29,145,192}
\definecolor{seq4}{RGB}{65,182,196}
\definecolor{seq5}{RGB}{127,205,187}
\definecolor{seq6}{RGB}{199,233,180}
\definecolor{seq7}{RGB}{237,248,177}
\definecolor{seq8}{RGB}{255,255,217}

\newcommand{\fcan}{f_\Omega}
\newcommand{\mucan}{{\mu_\Omega}}
\newcommand{\efcan}{\frac{1}{Z}e^{-\beta\left(\frac{p\cdot p}{2}+V(q)\right)}}
\newcommand{\fQ}{f_\mathcal{Q}}
\newcommand{\muQ}{{\mu_\mathcal{Q}}}

\newcommand{\pgen}[1]{G_{#1}}
\newcommand{\gen}{L}
\newcommand{\genl}{L_{\text{Lan}}}
\newcommand{\gens}{L_{\text{Smol}}}

\newcommand{\op}[1]{P^{#1}}
\newcommand{\opl}[1]{P_\text{Lan}^{#1}}
\newcommand{\ops}[1]{P_\text{Smol}^{#1}}

\newcommand{\Lw}[3]{\mathcal{L}^{#1}_{#2}\!\left({#3}\right)}
\newcommand{\C}[2]{\mathcal{C}^{#1}\left({#2}\right)}
\newcommand{\V}[2]{\mathcal{V}^{#1}\left({#2}\right)}
\newcommand{\W}[3]{\mathcal{W}^{#1}_{#2}\left({#3}\right)}

\newcommand{\LfQ}[1]{{\mathcal{L}^{#1}_{\mu_\mathcal{Q}}}}
\newcommand{\Qspace}{\mathcal{Q}}
\newcommand{\Pspace}{\mathcal{P}}
\newcommand{\Dom}[1]{\mathcal{D}\left({#1}\right)}

\theoremstyle{plain}
\newtheorem{theorem}{Theorem}[section]
\newtheorem{lemma}[theorem]{Lemma}
\newtheorem*{remark*}{Remark}
\newtheorem{remark}[theorem]{Remark}

\newtheorem{proposition}[theorem]{Proposition}
\newtheorem{corollary}[theorem]{Corollary}

\theoremstyle{definition}
\newtheorem{definition}[theorem]{Definition}
\newtheorem{assumption}[theorem]{Assumption}

\author{Andreas Bittracher\thanks{Center for Mathematics, Technische Universit\"at M\"unchen} \and P\'eter Koltai\thanks{Mathematics Institute, Freie Universit\"at Berlin} \and Oliver Junge\footnotemark[1]}

\title{Pseudo generators of spatial transfer operators}

\begin{document}
\maketitle

\begin{abstract}
Metastable behavior in dynamical systems may be a significant challenge for a simulation based analysis.  In recent years, transfer operator based approaches to problems exhibiting metastability have matured.  In order to make these approaches computationally feasible for larger systems, various reduction techniques have been proposed: For example, Sch{\"u}tte introduced a \emph{spatial transfer operator} which acts on densities on configuration space, while Weber proposed to avoid trajectory simulation (like Froyland et al.) by considering a discrete generator.  

In this manuscript, we show that even though the family of spatial transfer operators is not a semigroup, it possesses a well defined generating structure.  What is more, the \emph{pseudo generators} up to order 4 in the Taylor expansion of this family have particularly simple, explicit expressions involving no momentum averaging. This makes collocation methods particularly easy to implement and computationally efficient, which in turn may open the door for further efficiency improvements in, e.g., the computational treatment of conformation dynamics. We experimentally verify the predicted properties of these pseudo generators by means of two academic examples.
\end{abstract}

\section{Introduction}

\paragraph{Conformations of molecular systems}

The properties of many biomolecular systems such as proteins or enzymes depend heavily on their molecular configuration, i.e.\ the position of single atoms relative to each other. It is often observed that the system tends to "cluster" around certain key configurations. Transitions between these so-called \emph{conformations} can be considered rare events, as the time scale on which they occur and the characteristic dynamic time scale of the atoms in the molecule typically lie $10-15$ orders of magnitude apart. Nevertheless, these transitions play an essential role for the biological function of these molecules~\cite{ElKa87, Zhou98, FrMcM00, OWPN00, NoeEtAl06}. The reliable identification of these conformations and the probabilities (and rates) of transitions between them via direct numerical simulation is computationally very demanding if not infeasible for larger molecules.

\paragraph{Transfer operator based methods}

Molecular systems as described above are typically modeled as Hamiltonian systems, possibly including stochastic perturbations. Conformations then are almost-invariant (metastable) subsets of position space, corresponding to local minima of the potential energy surface. The ultimate goal of conformation analysis is to obtain a reduced model of the given system which accurately depicts these sets and the proper statistics of the transitions between them. This field of research, also called \emph{Markov state modeling}, attracted a lot of interest in the last decade~\cite{BHP09,ChoEtAl07,HuaEtAl10,PBB10,PriEtAl11,SchEtAl11,Web12}.

Pioneering work of Deuflhard, Dellnitz et al~\cite{Deu96, Deu01} exploits that these almost invariant sets can in principle be identified through eigenfunctions of a certain linear operator, the \emph{transfer operator}, which describes the evolution of distributions under the dynamics.

A direct application of this approach considers the operator acting on densities on the entire state space (i.e.\ position and momentum), while the conformational changes of interest are only observed in the position coordinate. Moreover, the approach is subject to the curse of dimension for all but the smallest systems, since a discretization of state space has to be constructed.

To remedy this, one might consider the ``overdampled'' or Smoluchowski dynamics~\cite{Hui01}, which acts on position space only, but this is a physically acceptable model of molecular motion only in the case that random collisions with the solvent overwhelm the effect of inertia in the molecule of interest. Sch\"utte~\cite{Sch99} came up with a physically justifiable solution as he introduced the so-called \emph{spatial dynamics}, whose metastable sets still bear the interpretation of molecular dynamical conformations. The associated \emph{spatial transfer operator} acts on densities on position space and can be seen as a momentum-averaged version of the full Hamiltonian or Langevin dynamics.

Commonly, the transfer operator is finitely approximated by a stochastic matrix, whose entries can be interpreted as transition probabilities of single system instances in the canonical ensemble from some subset (in state or position space) to another. These transition probabilities in turn are computed by short time integration of a number of trajectories starting in each subset.  Thus, the computation of a  few long simulations is replaced by the computation of many short trajectories for an ensemble of adequately distributed initial conditions.  Still, the momentum averaging has to be done explicitely by additionally sampling the momentum space for each of these initial conditions.

Over the years, different techniques for a finite approximation of transfer operators have been proposed, we refer to \cite{Del99,dellnitz2001algorithms} and the references there. More recently, Weber~\cite{Web06} used meshfree approximation techniques and showed that for a given approximation error, the number of basis elements scales with the number of metastable sets, not necessarily with the dimension of the sytem.  In \cite{Jun09}, an approximation by a sparse Haar space was proposed  in order to mollify the curse of dimension, while in \cite{Fri09} a tensor-product construction was used in combination with a mean field- approach. 
A novel approach to coarse-grain a multi-scale system by discretizing its transfer operator without using a full partition of the phase space~\cite{SchEtAl11} excels especially in efficiently reproducing the dominant time scales of the original system, but relies heavily on long trajectory simulations.

Elaborate schemes to extract the metastability information from the eigenvectors of the (approximate) transfer operator were developed in \cite{Hui01, Hui05, Deu04, HMSch04}. 

\paragraph{Simulation-free and generator-based methods}

All these methods rely on the numerical integration of trajectories. Only recently, methods have been proposed that require no time integration, albeit imposing further requirements on the system \cite{Kol10,Web12,Fro13}:  Under these conditions and provided that the system's transfer operator forms a continuous time semigroup, one can exploit that the eigenvalues and -functions are the same as those of the semigroup's \emph{infinitesimal generator}. Discretizing this generator requires no time-integration and is thus computationally considerably cheaper than classical methods.

\paragraph{This manuscript} Unfortunately, Sch\"utte's spatial transfer operator is not a time semigroup. In this contribution, we define suitable \emph{pseudo generators} of non-semigroup families of operators which inherit desirable properties of the spatial transfer operator as well as ``restored'' operators which approximate the spatial operator at least for small times. The appeal of these constructions from a numerical perspective is threefold:
\begin{enumerate}[(1)]
\item no numerical time integration is needed,
\item momentum averaging is accomplished analytically, i.e.\ momentum sampling is completely avoided,
\item the pseudo generators can be discretized by collocation methods, avoiding costly boundary integrals.
\end{enumerate}
We establish theoretical asymptotic estimates on the error of density propagation, and validate them numerically. The numerical experiments indicate that the information on metastable sets (i.e.\ conformations) gained from the restored operators remains close to the ``original'' one gained from the exact spatial transfer operator, even for times beyond those guaranteed by our estimates. A quantitative understanding of this phenomenon (also observed by Sch\"utte~\cite{Sch99}) is still lacking; some steps towards a theoretical explanation have been made in~\cite{BiHaKoJu15}. 

The manuscript is structured as follows. In Section~2, we introduce the basic dynamical models we are  working with and describe their action on propagating (probability) densities by transfer operators. This necessitates the discussion of operator semigroups, given at the end of the section. Section~3 is concerned with fluctuations in the spatial distribution of the system governed by its dynamics, leading to the concept of spatial transfer operators as well as metastability in position space. In Section~4, we introduce the concept of pseudo generators, the corresponding restored operators, and give asymptotic error estimates on their approximation quality. Section~5 includes numerical experiments. We conclude our work in Section~6, and discuss future directions to make the method applicable for realistic \mbox{(bio-)}molecular systems. 
Three appendices are given: Appendix~\ref{proofs} gives a detailed derivation of the pseudo generators up to order~3; Appendix~\ref{sec:spatial_spectrum} gives a complete and self-contained proof of the applicability of Huisinga's theory~\cite{Hui01, Hui05} on the quantitavive identification of metastable components from spectral analysis of transfer operators for the spatial transfer operator based on Langevin dynamics (especially reversibility and ergodicity of the spatial dynamics, which are probably known or at least anticipated, however we could not find neither a statement, nor even a partial derivation of these properties); in Appendix~\ref{sec:smoothness}, we show that the eigenfunctions of the spatial transfer operator are smooth, i.e.\ infinitely differentiable, if the potential is a smooth function.

\section{Transfer operators and their generators}

In this section we introduce the dynamical systems of interest as well as the concept of transfer operators for describing statistical transport under these dynamics.

\subsection{Stochastic dynamics}\label{stochasticdynamics}

Broadly speaking, we will be studying continuous time stochastic dynamical systems on a phase space $\Omega\subset\mathbb{R}^d$. They will be described by $\Omega$-valued random variables $\bm{x}_t$, $t\ge0$, following an It\^o diffusion equation, i.e. a stochastic differential equation of the form
\begin{equation}\label{itodiffusion}
\partial_t \bm{x}_t=b(\bm{x}_t)+\Sigma(\bm{x}_t)\bm{w}_t\,.
\end{equation}
Here, $\partial_t$ denotes the differentiation with respect to~$t$, $b:\Omega\rightarrow \mathbb{R}^d,~\Sigma:\Omega\rightarrow\mathbb{R}^{d\times d}$ and~$\bm{w}_t$ is a $\mathbb{R}^d$-valued ``white noise'' term, see e.g.~\cite[p.~61]{Oks98}. The functions $b$ and $\Sigma$ are assumed to be globally Lipschitz and growing at most linearly at infinity, such that~\eqref{itodiffusion} has unique solutions (cf.~\cite[Theorem~5.2.1]{Oks98}). Here and in the following boldface lower case letters denote random variables. 

\paragraph{Molecular dynamics}  Consider a molecular system described by $d\in\mathbb{N}$ positional degrees of freedom. Typically, these correspond to either internal coordinates or particle positions in $\mathbb{R}^{3n}$, $n$ being the number of particles. Let $\Qspace\subset\mathbb{R}^d$ denote the configuration space.

Let a potential $V:\Qspace\rightarrow \mathbb{R}$, describing the relative energy of a given configuration~$q\in\Qspace$, be continously differentiable\footnote{Later on, we will impose stronger assumptions on~$V$.}. Under the assumption of strict total energy, the movement of the system is then described by classical deterministic Hamiltonian dynamics:
\begin{equation}\label{hamiltondynamics}
\begin{aligned}
\partial_t q&=p\\
\partial_t p&=-\nabla V(q).
\end{aligned}
\end{equation}
The phase space is thus $\Omega=\Qspace\times\Pspace$, with $q\in\Qspace$ the position and $p\in\Pspace=\mathbb{R}^d$ the momentum coordinate. For simplicity, we set the mass matrix~$M$ to be the identity, otherwise the first equation in~\eqref{hamiltondynamics} would be $\partial_t q = M^{-1}p$.

Equation~\eqref{hamiltondynamics} models a system ``in vacuo'', independent from external influence. Of more physical relevance, however, are systems which are stochastically coupled to their surroundings, physically motivated by the presence of a heat bath or implicit solvent not modeled explicitly. A prominent way of doing this is via a drift-diffusion perturbation of (\ref{hamiltondynamics}), known as the Langevin equations, which can be formally derived using averaging techniques from the Mori-Zwanzig formalism~\cite{Zwa01},
\begin{equation}\label{langevindynamics}
\begin{aligned}
\partial_t\bm{q}_t &= \bm{p}_t\\
\partial_t\bm{p}_t &= -\nabla V(\bm{q}_t) - \gamma \bm{p}_t +\sigma \bm{w}_t~.
\end{aligned}
\end{equation}
These can be written in the form of (\ref{itodiffusion}) with
$$
b(q,p)=\begin{pmatrix}p\\ -\nabla V(q)-\gamma p\end{pmatrix}~~\text{and}~~\Sigma(q,p) = \begin{pmatrix}0&0\\0& \sigma\end{pmatrix}.
$$
The term $-\gamma \bm{p}_t$ mimics the drag through the implicitly present solvent, $\sigma \bm{w}_t$ accounts for random collisions with the solvent particles. To balance damping and excitation and to keep the system at a constant average internal energy\footnote{Actually, $\beta$ is the inverse temperature, $\beta=1/(k_BT)$, with Boltzmann's constant $k_B$ and $T$ the system temperature.} $\beta$, we set $\sigma=\sqrt{2\gamma/ \beta}$. The choice of~$\gamma$ is problem dependent, and mimics the viscosity of the aforementioned implicit solvent. For further details on the modeling see~\cite{Cra04}.

An even further model reduction leads to the so-called Smoluchowski dynamics. In the second-order form of (\ref{langevindynamics}),
\begin{equation}\label{secondorderlangevindynamics}
\partial_t^2\bm{q}_t=-\nabla V(\bm{q}_t)-\gamma \partial_t\bm{q}_t+\sigma \bm{w}_t\,,
\end{equation}
we consider a high-friction situation $\gamma\to\infty$. After appropriate rescaling of the time, $\tau=\gamma^{-1} t$ (in order to be able to observe movement under the now extremely slow dynamics), (\ref{secondorderlangevindynamics}) becomes
$$
\gamma^{-2}\partial_{\tau}^2\bm{q}_{\tau} = -\nabla V(\bm{q}_{\tau})- \partial_{\tau}\bm{q}_{\tau}+\sqrt{\frac{2}{\beta}} \bm{w}_{\tau}.
$$
In the limit $\gamma\to\infty$, this yields the \emph{Smoluchowski equation}
\begin{equation}\label{smoluchowskidynamics}
\partial_{\tau}\bm{q}_{\tau} = -\nabla V(\bm{q}_{\tau}) + \sqrt{\frac{2}{\beta}} \bm{w}_{\tau}.
\end{equation}
This conceptual derivation can be made precise by considering stochastic convergence notions. The interested reader is referred either to~\cite{Nel63}, where the physical intuition and mathematical rigor are both kept at a high level\footnote{A reader with a physicists view might be irritated by the fact that the terms with~$\nabla V$ and~$\bm w_t$ act as forces (or accelerations) in \eqref{langevindynamics} and as velocities in~\eqref{smoluchowskidynamics}. By eliminating the damping coefficient~$\gamma$, the physical dimensions of the terms changed. Since the mathematical statement is not affected by this, we shall not go into details.}, or to~\cite{Pav08}, where homogenization techniques for the transfer operators of the underlying equations are exploited. Next, we will consider this latter, \emph{operator-based} characterization of stochastic processes.

\subsection{Transfer operators}\label{ssec:transferoperator}

We shall now examine how \emph{phase space density functions} evolve under the dynamics induced by~\eqref{itodiffusion}. That is, given a probability density at time $t_0=0$, what is the probability to find the system in a certain region at time $t>t_0$? More precisely, given $\bm x_0 \sim f = f_0$ (a random variable $\bm x_0$ distributed according to the density~$f_0$), find $f_t$ with $\bm x_t\sim f_t$, $t\ge0$, where the evolution of $\bm x_t$ is governed by~\eqref{itodiffusion}.



To this end, let $(\Omega,\mathcal{B},\mu)$ be a probability space with $\mathcal{B}$ denoting the Borel $\sigma$-algebra, and consider the \emph{stochastic transition function} $p:\mathbb{R}_{\ge0}\times\Omega\times\mathcal{B}\to [0,1]$,
\[
p(t,x,B) = \mathrm{Prob}\left[ \bm x_t\in B\,\vert\, \bm x_0=x\text{ with probability 1}\right],
\]
and denote by
\begin{equation}\label{transitionprobability}
p_{\mu}(t,A,B) := \mathrm{Prob}_{\mu}\left[\bm x_t\in B\,\vert\,\bm x_0\in A\right]
\end{equation}
the \emph{transition probabilities} between $A\in\mathcal{B}$ and $B\in\mathcal{B}$, where $\mathrm{Prob}_{\mu}$ indicates that $\bm x_0\sim\mu$; i.e.\ the initial condition is distributed according to~$\mu$. For the long term macroscopic behavior of the system, sets $A\subset\Omega$ play an important role for which~$p_{\mu}(t,A,A)\approx 1$ for some physically relevant measure~$\mu$ and times~$t>0$.

Now assume that an initial distribution $\bm x_0\sim f =: f_0\in\Lw{1}{\mu}{\Omega}$\footnote{In the literature, $\mathcal{L}^p$ sometimes denotes the ``pre-Lebesgue space'', i.e. the Lebesgue space before equivalence class formation, and $L^p$ usually denotes the actual Lebesgue space. Due to clash of notation, however, we call the actual Lebesgue space $\mathcal{L}^p$ and use $\|\cdot\|_{k,\mu}$ to denote the standard norm.} is given. 
We then have that $\bm x_t\sim f_t$ with
\begin{equation}\label{stochtransport}
\int_Bf_t(x)~d\mu(x) = \int_\Omega f_0(x)p(t,x,B)\,d\mu(x),\qquad \forall\ B\in\mathcal{B}.
\end{equation}
Under mild conditions\footnote{See e.g.~\cite{Las94}.}, satisfied by the systems considered here, $f_t$ is uniquely defined by~\eqref{stochtransport}. This yields the \emph{transfer operator with lag time~$t$} $P^t:\Lw{1}{\mu}{\Omega}\to \Lw{1}{\mu}{\Omega}$ via
\[
\op{t} f(x) := f_t(x)
\]
where we extend the definition of $\op{t}$ from densities to arbitrary integrable functions using linearity. In the deterministic case, this is the so-called \emph{Perron--Frobenius operator}. 

For some of the following results, it will be necessary to distinguish between the transfer operators of the general It\^o diffusion (\ref{itodiffusion}), and the special cases of the Langevin (\ref{langevindynamics}) and Smoluchowski dynamics (\ref{smoluchowskidynamics}). We will then refer to them as $\op{t}$, $\opl{t}$ and $\ops{t}$, respectively. Of course, statements concerning $\op{t}$ hold for $\opl{t}$ and $\ops{t}$ as well.

\paragraph{Some properties of $P^t$}
$\op{t}$ can be considered a time-parametrized family of linear operators, which then possesses the Chapman--Kolmogorov (or semigroup) property:

\begin{enumerate}
\item[(i)] $\lim_{t\rightarrow 0}P^tf=f$,
\item[(ii)] $\op{t+s}f=\op{t}\big(\op{s}f\big)$ for all $s,t\ge0$.
\end{enumerate}

While a stochastic interpretation only makes sense in the preceding setting, the formal extension of $\op{t}$ to the spaces~$\Lw{k}{\mu}{\Omega}$, $1\le k\le \infty$, is well defined for proper choices of~$\mu$ (see Corollary \ref{cor:opcontraction}).
We thus have that $\|\op{t}\|_{k,\mu} \le 1$ and  $\op{t}f\ge 0$ for $0\le f\in \Lw{k}{\mu}{\Omega}$.

Using the standard scalar product on $\Lw{2}{\mu}{\Omega}$, for $\mu(A)>0$, the transition probabilities can be expressed via the transfer operator:
\[
p_{\mu}(t,A,B) = \frac{1}{\mu(A)}\int_B \op{t}\chi_{A}\,d\mu = \frac{1}{\mu(A)}\int_\Omega \op{t}\chi_{A}\chi_{B}\,d\mu=\frac{\langle P^t\chi_A,\chi_B\rangle_{2,\mu}}{\langle \chi_A,\chi_A\rangle_{2,\mu}}
\]
with~$\chi$ being the indicator function.

\subsection{Infinitesimal generators}

The semigroup property basically means that~$\op{t}$ is ``memoryless'' (in other terms, \eqref{itodiffusion} generates a Markov process). The identity $\op{t} = \left(\op{t/n}\right)^n$ suggests that all information about the density transport is contained in $P^\tau$ for \emph{arbitrarily small} $\tau$.

This is formalized by looking at an operator $\gen:\Dom{\gen}\rightarrow \Lw{k}{\mu}{\Omega}$ given by
\begin{equation}\label{generator}
\gen f=\lim_{\tau\rightarrow0}\frac{\op{\tau}f-f}{\tau},
\end{equation}
where $\Dom{\gen}\subset \Lw{k}{\mu}{\Omega}$ is the linear subspace of $\Lw{k}{\mu}{\Omega}$ where the above limit exists. $\gen$ is called the \emph{infinitesimal generator} of the semigroup $\op{t}$, and the field of operator semigroup theory~\cite{Paz83} answers the question in which sense $\op{t}$ is a solution operator to the Cauchy problem $\partial_t f_t = \gen f_t$. Essentially, the power of the infinitesimal generator lies in the fact that all the relevant information about~$\op{t}$ for \emph{all times $t\ge0$} is already encoded in~$\gen$. We will discuss this below.

\paragraph{Invariant density}

Having its interpretation in mind, it is not surprising that the infinitesimal generator is exactly the right hand side of the parabolic partial differential equation describing the flow of sufficiently regular densities, the \emph{Kolmogorov forward equation} or \emph{Fokker--Planck equation}, see \cite[p.~282]{Kar91}:
\begin{equation}\label{fokkerplanckequationito}
\partial_tf_t(x) = \underbrace{\frac{1}{2}\sum_{i=1}^d\sum_{k=1}^d\frac{\partial^2}{\partial x_i\partial x_k}\big(\Sigma_{ik}(x)f_t(x)\big) - \sum_{i=1}^d\frac{\partial}{\partial x_i}\big(b_i(x)f_t(x)\big)}_{=:\gen f_t(x)}.
\end{equation}
In the Langevin case, this simplifies to
\begin{equation}\label{fokkerplanckequationlangevin}
\partial_t f_t(q,p) = \underbrace{\Big(\frac{\gamma}{\beta}\Delta_p-p\cdot\nabla_q+\nabla_qV(q)\cdot\nabla_p+\gamma p\cdot\nabla_p+d\gamma\Big)}_{=:\genl}f_t(q,p)\, ,
\end{equation}
where the dot denotes the Euclidean inner product, $\nabla_x$ and $\Delta_x$ are the gradient and Laplace operators with respect to~$x$, respectively. For Smoluchowski, it is
\begin{equation}\label{fokkerplanckequationsmoluchowski}
\partial_t f_t(q) = \underbrace{\Big( \frac{1}{\beta}\Delta_q+\nabla_qV\cdot\nabla_q +\Delta_qV \Big)}_{=:\gens}f_t(q).
\end{equation}

Densities which are invariant under the dynamics play a naturally prominent role. Since they are fixed under~$\op{t}$ for any $t\ge0$, by~\eqref{generator} they lie in the kernel of~$\gen$; see also Corollary~\ref{cor:invdens} below. For the stochastic processes considered here the invariant density can be shown to be unique (cf.~\cite{Oks98}), and---using the term from statistical mechanics---we call it the \emph{canonical density},~$\fcan$. For the Langevin dynamics 
\begin{equation}\label{canonicaldensity}
\fcan(q,p)=f_\Qspace(q)\cdot f_\Pspace(p),
\end{equation}
where
\begin{equation*}
	f_\Qspace(q) := \frac{1}{Z_\Qspace}\exp\big(-\beta V(q)\big), \qquad 	f_\Pspace(p) := \frac{1}{Z_\Pspace}\exp\big(-\beta\frac{p\cdot p}{2}\big),
\end{equation*}
with $Z_\Pspace=\int_\Pspace\exp\big(-\beta\frac{p\cdot p}{2}\big)dp$ and $Z_\Qspace=\int_\Pspace\exp\big(-\beta V(q)\big)dq$. For the Smoluchowski dynamics, $f_\Omega(q)=f_{\Qspace}(q)$ is the canonical density. $f_\Omega$ is a density with respect to the Lebesgue measure~$m$, and its existence requires the integrability of~$\exp(-\beta V)$; which is from now on assumed to hold. 

To understand their relevance, note that the canonical density not just describes the statistical equilibrium of the system, but the system also tends to this equilibrium as time grows, i.e.\ according to whatever $f_0$ the system is distributed initially, $f_t\to\fcan$ as $t\to\infty$ in~$\Lw{1}{m}{\Omega}$.

\paragraph{Domain and Spectral properties}

The results of this paragraph hold true for the transfer operators of both the Langevin and the Smoluchowski dynamics, by taking the corresponding phase space and invariant measure (i.e.\ the measure having the canonical density as Radon--Nikod\'ym derivative with respect to the Lebesgue measure).

Let $\mucan$ be the invariant measure of $P^t$. Note that for arbitrary \mbox{$1\leq k\leq\infty$}, $\op{t}$ can be defined on $\Lw{k}{\mucan}{\Omega}$ due to the following corollary:

\begin{corollary}[{\cite[Corollary to Lemma 1]{BaRo95}}]\label{cor:opcontraction}
Let $P^t$ be a transfer operator associated with a transition function having the invariant measure $\mucan$. Then $P^t$ is a well-defined contraction on $\mathcal{L}^k_\mucan(\Omega)$ for every $1\leq k\leq\infty$.
\end{corollary}

For our purposes the main connection between a semigroup of operators and their generator is given by the following
\begin{theorem}[Spectral mapping theorem~\cite{Paz83}]\label{spectralmappingtheorem} Let $\mathcal{X}$ be a Banach space, $T^t:\mathcal{X}\to\mathcal{X}$, $t\ge0$, a $C_0$ semigroup of bounded linear operators (i.e.\ $T^tf\to f$ as $t\to0$ for every $f\in\mathcal{X}$, and~$T^t$ bounded for every~$t$), and let $A$ be its infinitesimal generator. Then
$$
e^{t\sigma_p(A)}\subset \sigma_p(T^t)\subset e^{t\sigma_p(A)}\cup \{0\},
$$
with $\sigma_p$ denoting the point spectrum. The corresponding eigenvectors are identical.
\end{theorem}

We can immediately deduce the following statements.

\begin{corollary}	\label{cor:invdens}
A function $f$ is an invariant density of $\op{t}$ for all $t\geq0$, if and only if $\gen f=0$.
\end{corollary}

\begin{corollary}
Since $\op{t}$ is a contraction in $\Lw{k}{\mucan}{\Omega}$, the eigenvalues of $\gen$ lie in the left complex half-plane.
\end{corollary}

Theorem \ref{spectralmappingtheorem} suggests that $\op{t}=e^{t\gen}=\sum_{k=0}^\infty \frac{t^k}{k!}\gen^{k}$. This intuition is false in general, as $\gen$ may be unbounded and $\bigcap_{k=1}^{\infty}\Dom{\gen^k} \neq \Lw{p}{\mucan}{\Omega}$ for any~$p$. However, $\op{t}$ can be approximated by a truncated ``Taylor series'', at least pointwise, in the function space
\begin{equation}\label{Vspace}
\displaystyle \V{N}{\Omega}:=\big\{f\in\C{2N}{\Omega}~\big\vert~\gen^{n}f\in \Lw{2}{\mucan}{\Omega}~\forall n=0,\ldots,N\big\}.
\end{equation}
We require~$f$ and~$V$ to be $2N$-times differentiable, as this is the highest derivative occuring in~$\gen^{N}$, cf.~(\ref{fokkerplanckequationito}).

The following convergence result also holds true if choosing $\Lw{k}{\mucan}{\Omega}$ instead of $\Lw{2}{\mucan}{\Omega}$ in the definition of $\V{N}{\Omega}$, and correspondingly regarding the norm $\|\cdot\|_{k,\mu_\Omega}$. However, we state it for $\Lw{2}{\mucan}{\Omega}$, as this is the space we are ultimately operating in.

\begin{proposition}\label{truncatedtaylor}
Let $f\in \V{N+1}{\Omega}$. Then
$$
\Big\|\op{t}f-\sum_{n=0}^N\frac{t^n}{n!}\gen^{n}f\Big\|_{2,\mucan} = \mathcal{O}(t^{N+1}) ~~ \text{for $t\rightarrow 0$}.
$$
\end{proposition}

\begin{proof}
Let $f\in\V{N+1}{\Omega}$. Then, $\op{t} f: t\mapsto \Lw{1}{\mucan}{\Omega}$ is $N+1$ times differentiable in~$t$ because $\partial^k\op{t}\big|_{t=0}f=\gen^{k}f,~k\leq N+1$ exist as per choice of $f$. The Taylor series expansion for Banach space valued linear operators can for example be found in \cite[Section 4.5]{Zei95}. Application to $\op{t}$ yields
\begin{align*}
\op{t}f=\sum_{n=0}^N\gen^{n}f\frac{t^n}{n!}+ \Big(\int_0^1\frac{1}{N!}(1-s)^N\partial_s^{N+1}\op{st}f~ds\Big)t^{N+1}.
\end{align*}
We estimate the remainder:
\begin{align*}
\Big\|\op{t}f-\sum_{n=0}^N\frac{t^n}{n!}\gen^{n}f\Big\|_{2,\mucan} &= \Big\|\int_0^1\frac{1}{N!}(1-s)^N\partial_s^{N+1}\op{st}f~ds\Big\|_{2,\mucan}t^{N+1}.
\intertext{As we are interested in the limit $t\rightarrow 0$ we can assume $t<1$. In that case, $st<s$, and therefore}
&\leq\frac{t^{N+1}}{N!}\sup_{s\in[0,1]}\big\|\partial_s^{N+1}\op{s}f\big\|_{2,\mucan}.
\intertext{$\op{t}f$ is the solution of the Fokker--Planck equation (\ref{fokkerplanckequationito}), and thus $\partial_s \op{s}f =\gen\op{s}f$ and, by extension, \mbox{$\partial_s^{N+1}\op{s}f=\gen^{N+1}\op{s}f$}. Moreover, due to Pazy \cite[Corollary 1.4]{Paz83}, the transfer operator and generator commute: \mbox{$\gen\op{s}f=\op{s}\gen f$}. Therefore,}
&=\frac{t^{N+1}}{N!}\sup_{s\in[0,1]}\big\|\op{s} \gen^{N+1} f\big\|_{2,\mucan} \\
&\leq \frac{t^{N+1}}{N!} \sup_{s\in[0,1]} \underbrace{\big\|\op{s}\big\|_{2,\mucan}}_{\leq 1} \big\|\gen^{N+1} f\big\|_{2,\mucan}.
\end{align*}
In the last line, $\big\|\op{s}\big\|_{2,\mucan}\leq 1$ because $P^t$ is a contraction (Corollary \ref{cor:opcontraction}). As $\gen^{N+1} f \in \Lw{2}{\mucan}{\Omega}$ by the choice of~$f$, it is bounded. This completes the proof. 
\end{proof}

\section{Spatial dynamics and metastability}

In order to analyze the behaviour of molecular systems in regard of configurational stability, we have to restrict our view to the dynamics on position space $\Qspace$. For this purpose, Sch\"utte in \cite{Sch99} proposed a reduction of the classical Hamiltonian dynamics, called Hamiltonian dynamics with randomized momenta, while Weber \cite{Web12} proposed the corresponding generalized version for a stochastic evolution. Following Sch\"utte and Weber, we formulate the extension to Langevin dynamics and state the appropriate definition of metastability.

\subsection{The spatial transfer operator}

Consider an infinitely large number of identical systems of form (\ref{langevindynamics}) in thermodynamic equilibrium, i.e. identically and independently distributed according to $\fcan$ (called an \emph{ensemble} in classical statistical mechanics literature). To determine which portion of these systems undergo a certain configurational change, i.e. leave a subset $A\subset\Qspace$, we have to track the evolution of all these systems starting from $A$. Due to the product structure (\ref{canonicaldensity}) of $\fcan$, their momenta are still distributed according to $f_\mathcal{P}$ and so the whole coordinates are initially distributed according to $\chi_A\fcan$, with $\chi_A$ the indicator function of $A$ on $\Qspace$. This phase space density now evolves under $\opl{t}$, but as we are only interested in the position portion of the evolving density, we form the marginal distribution with respect to $q$.  The resulting \emph{spatial transfer operator} on $\LfQ{k}\left(\Qspace\right)$ with $d\muQ:=\fQ dm$ is
\begin{equation} \label{spatialto}
S^t\chi_A(q):=\frac{1}{\fQ(q)}\int_\Pspace\opl{t}\big(\chi_A(q)\fcan(q,p)\big)~dp,
\end{equation}
cf.\ Corollary \ref{cor:opcontraction} applied to the operator $\opl{t}$ and the  invariant measure $\mucan$.

Intuitively, one can think of $S^tu$ with normalized $u\in\LfQ{k}\left(\Qspace\right)$ as transporting a positional portion of the canonical density.

\subsection{Metastability on position space}

\paragraph{Langevin dynamics with randomized momenta}
Considering $S^t$ as an operator on $\LfQ{2}(\Qspace)$ and using the standard associated scalar product $
\langle u,v\rangle_{2,\muQ}$ gives us access to certain transition probabilities on $\Qspace$, which fit our intuition of metastability. For $A\subset\Qspace$ we call
$$
\Gamma(A) :=\big\{(q,p)\in\Omega~|~q\in A\big\}
$$
the ''slice`` of phase space corresponding to $A$. It represents a sub-ensemble in position space associated to \emph{all} possible momenta. It is easy to see that the transition probabilities between slices $\Gamma(A)$ and $\Gamma(B)$ can be expressed in terms of $S^t$:
\begin{equation}
\label{slicetransitionprob}
 p\left(t,\Gamma(A),\Gamma(B)\right)=\frac{\langle S^t\chi_A,\chi_B\rangle_{2,\muQ}} {\langle\chi_A,\chi_A\rangle_{2,\muQ}},
\end{equation}
where~$p$ is the stochastic transition function under Langevin dynamics with respect to the Lebesgue measure. We now call a disjoint decomposition $A_1\cup\ldots\cup A_n=\Qspace$ of position space \emph{metastable} if
$$
p\left(t,\Gamma(A_j),\Gamma(A_j)\right)\approx 1,~j=1,\ldots,n.
$$
The meaning of ``$\approx 1$'' will become apparent later.

The connection between eigenvalues close to one of some transfer operator and metastable sets was first observed in~\cite{Del99} and applied in conformation dynamics in \cite{Deu96}.  An extension to a broader class of transfer operators (satisfying an assumption related to self-adjointness) was provided by Huisinga and Schmidt \cite{Hui05}. Our $S^t$ falls into that class, which is shown in Appendix \ref{sec:spatial_spectrum}.

\begin{theorem}[Application of {\cite[Theorem 2]{Hui05}}]\label{metastability}
Let $\sigma(S^t)\subset [a,1]$ with $a>-1$ and \\\mbox{$\lambda_n\leq\ldots\leq\lambda_2<\lambda_1=1$} be the $n$ largest eigenvalues of $S^t$, with eigenvectors $v_n,\ldots,v_1$.
Let $\{A_1,\ldots,A_n\}$ be a measurable decomposition of $\Qspace$ and $\Pi: \LfQ{2}(\Qspace)\rightarrow \LfQ{2}(\Qspace)$ be the orthogonal projection onto $\operatorname{span}(\chi_{A_1},\ldots,\chi_{A_n})$, i.e. 
$$\Pi v = \sum_{j=1}^n\frac{\langle v,\chi_{A_j}\rangle_{2,\muQ}}{\langle \chi_{A_j},\chi_{A_j}\rangle_{2,\muQ}}\chi_{A_j}~.$$

The metastability of the decomposition can then be bounded from above by
$$
p\left(t,\Gamma(A_1),\Gamma(A_1)\right)+\ldots + p\left(t,\Gamma(A_n),\Gamma(A_n)\right) \leq 1+\lambda_2+\ldots+\lambda_n,
$$
while it is bounded from below by
$$
1+\rho_2\lambda_2 +\ldots+\rho_n\lambda_n + c \leq p\left(t,\Gamma(A_1),\Gamma(A_1)\right)+\ldots p\left(t,\Gamma(A_n),\Gamma(A_n)\right)
$$
where $\rho_j=\| \Pi v_j\|_{2,\muQ}\in [0,1]$ and $c= a(1-\rho_2+\ldots+1-\rho_n)$.
\end{theorem}

Thus, the lower the projection error of $\Pi v_j$, the better the lower bound matches the upper bound.
We choose $A_1,\ldots,A_n$ in accordance to the the sign structure of $v_1,\ldots,v_n$ as a heuristic to the 
optimal decomposition (i.e. we treat the eigenfunctions as one-step functions of the form $\chi_A-\chi_B$).

However, more sophisticated strategies for extracting metastable sets are available, most notably the linear optimization-based PCCA-algorithm and its extensions, developed by Deuflhard et. al. (\cite{Deu00}, \cite{Roe13}).
Let it be noted that it is applicable to all the operators developed herein, as PCCA does not depend on the underlying dynamical model.

\paragraph{Smoluchowski dynamics}

As described in section \ref{stochasticdynamics}, another way to restrict the molecular dynamics to position space is via the high-friction limit and the arising transition from Langevin to Smoluchowski dynamics. As this limit may represent a considerable deviation from physical reality, it is initially unclear how metastability in system \eqref{smoluchowskidynamics} can be interpreted in the context of the original system. As the transition also involves a rescaling of time, especially the transition probabilities have to be treated with caution.

Nevertheless, metastability under Smoluchowski dynamics can formally be defined as above, and it holds
$$
 p\left(t,A,B\right)=\frac{\langle \ops{t}\chi_A,\chi_B\rangle_{2,\muQ}} {\langle\chi_A,\chi_A\rangle_{2,\muQ}}.
$$
Here, $p$ is the stochastic transition function with respect to Smoluchowski dynamics. Using the same reasoning as in the previous paragraph, we seek eigenpairs~$(\lambda,u)$ of~$\ops{t}$ with~$\lambda\approx 1$.

However, in the case of Smoluchowski dynamics, these are somewhat more accessible. Due to the Spectral Mapping Theorem \ref{spectralmappingtheorem}, eigenvalues of $\ops{t}$ near $1$ coincide with eigenvalues of the infinitesimal generator $\gens$ near $0$, and the associated eigenvectors are identical. An efficient method for metastability analysis based on $\gens$ was developed in \cite {Fro13}. There it is shown that for an eigenvalue $\lambda <0$ of $\gens$, the corresponding eigenvector $u$ and the sets $A^+ =\{u> 0\},~A^-=\{u<0\}$ holds
$$
p(t,A^+,A^+) + p(t,A^-,A^-) = 1+ \operatorname{exp}(t\lambda) +\mathcal{O}(t).
$$


Unfortunately, $S^t$ lacks the semi-group property, and so cannot be the solution operator of an autonomous PDE, such as the Fokker--Planck equation. Equivalently, spatial dynamics is not induced by an It\^o diffusion process, and thus has no infinitesimal generator in the sense of (\ref{fokkerplanckequationsmoluchowski}).

\section{The generating structure of spatial transfer operators}

Formally, the time-derivatives of $S^t$ can still be defined, in analogy to (\ref{generator}). We will see in the following how the resulting operators can play the role of the infinitesimal generator in the context of metastability analysis.
 
\subsection{Pseudo generators}

We first define these time derivatives for general time-parameterized operators:

\begin{definition}\label{pseudogenerator}
Let $\mathcal{X}$ be a Banach space, $T^t:\mathcal{X}\rightarrow \mathcal{X},~ t>0$ be a time-parametrized family of bounded linear operators. 

Define the operator $\partial_tT^t:\Dom{\partial_tT^t}\rightarrow\mathcal{X}$ by
$$
\partial_tT^tf = \underset{h\rightarrow 0}{\lim}\frac{T^{t+h}f-T^tf}{h}
$$
and call it the \emph{time-derivative} of $T^t$. $\Dom{\partial_tT^t}$ here is the subspace of $\mathcal{X}$ where the above limit exists.
Iteratively, we define by
$
\partial_t^{n}T^t:=\partial_t\big(\partial_t^{n-1}T^t\big)
$
the $n$-th time-derivative on $\Dom{\partial_t^nT^t}$.
Finally,
$$
\pgen{n} := \partial_t^n T^t\big|_{t=0}
$$
is called the $n$-th \emph{pseudo generator} of $T^t$.
\end{definition}
For $T^t=\op{t}$, the transfer operator of an It\^o process, the pseudo generators are the iterated infinitesimal generators:
\begin{proposition}
On $\mathcal{D}(L^n)$, the $n$-th pseudo generator $\pgen{n}$ of $\op{t}$ takes the form
$$
\pgen{n}=\gen^{n},
$$
with $\gen$ the infinitesimal generator of the respective dynamics.
\end{proposition}

With this, the pseudo generators of the spatial transfer operator $S^t$ can be expressed by the generator $\genl$ of the full Langevin transfer operator:

\begin{lemma}\label{altpseudogen}
On $\mathcal{D}(\genl^n)$, the $n$-th pseudo generator $\pgen{n}$ of $S^t$ takes the form
$$\pgen{n}u(q)=\frac{1}{\fQ(q)}\int_\Pspace\genl^n\big(u(q)\fcan(q,p)\big)~dp.$$
\end{lemma}
\begin{proof} The first time-derivative of $S^t$ is
\begin{align*}
\partial_tS^tu(q) &= \partial_t\Big(\frac{1}{\fQ(q)}\int_\Pspace\opl{t}\big(u(q)\fcan(q,p)\big)~dp\Big)\\
&=\frac{1}{\fQ(q)}\int_\Pspace\partial_t\opl{t}\big(u(q)\fcan(q,p)\big)~dp\\
&\underset{(\ref{fokkerplanckequationlangevin})}{=}\frac{1}{\fQ(q)}\int_\Pspace\Big(\genl \opl{t}\big(u(q)\fcan(q,p)\big)\Big)~dp.
\end{align*}
Inductively, this gives the $n$-th time derivative
$$
\partial_t^{n}S^t =\frac{1}{\fQ(q)}\int_\Pspace\genl^{n}\opl{t}\big(u(q)\fcan(q,p)\big)~dp.
$$
As $\opl{0}=\operatorname{id}$, the $n$-th pseudo generator is
$$
\pgen{n}u(q)= \Big(\frac{1}{\fQ(q)}\int_\Pspace\genl^n\opl{t}\big(u(q)\fcan(q,p)\big)~dp\Big)\Big|_{t=0} = \frac{1}{\fQ(q)}\int_\Pspace\genl^n\big(u(q)\fcan(q,p)\big)~dp.
$$
\end{proof}

From now on, when speaking of pseudo generators, we always mean pseudo generators of $S^t$. Note that, in general, $\pgen{n}$ is not simply a power of $\pgen{1}$, as the integral and the power of $\genl^n$ do not commute. We thus take a closer look at the first few $\pgen{n}$:

\begin{proposition}\label{gens}
Let $S^t$ be the spatial transfer operator for the Langevin dynamical process. On their respective domain, its first three pseudo generators take the form
\begin{enumerate}
\item $\displaystyle \pgen{1}=0$,
\item $\displaystyle \pgen{2}=\frac{1}{\beta}\Delta -\nabla V\cdot\nabla$. Notably, $\pgen{2}$ is independent of $\gamma$.
\item $\displaystyle \pgen{3}=-\gamma\pgen{2}$.
\end{enumerate} 
\end{proposition}
The proof can be found in Appendix \ref{proofs}.

\paragraph{Connection to Smoluchowski dynamics}

We can draw a perhaps surprising connection between the second pseudo generator and Smoluchowski dynamics. Recall that (\ref{fokkerplanckequationsmoluchowski}) and thus $\ops{t}$ deals with densities with respect to Lebesgue measure $m$. However, to compare it to $S^t$, we have to track the transport of densities with respect to $\muQ$. This transport is described by
\begin{equation}\label{kolmogorovforwardequationsmoluchowski}
\partial_tu_t(q) = \underbrace{\left(\frac{1}{\beta}\Delta-\nabla V\cdot\nabla\right)}_{=:G_\text{Smol}}u_t(q).
\end{equation}
Note, however, that the transition from $m$ to $\muQ$ is merely a basis transformation: Let, for a brief moment, $P^t_{\text{Smol},m}$ and $P^t_{\text{Smol},\muQ}$ be the transfer operators generated by $L_\text{Smol}$ and $G_\text{Smol}$, respectively. Then,
\begin{equation}\label{eq:smoltransformation}
P^t_{\text{Smol},m}\left(uf_\mathcal{Q}\right) = \left(P^t_{\text{Smol},\muQ}\left(u\right)\right) f_\mathcal{Q}.
\end{equation}

$P^t_{\text{Smol},\muQ}$ exists on every $\LfQ{k}(\Qspace)$,~$1\leq k\leq \infty$ (choosing $\mu=\fQ dm$ in Corollary \ref{cor:opcontraction}), and thus because of \eqref{eq:smoltransformation}, $P^t_{\text{Smol},m}$ exists on $\mathcal{L}_m^k(\Qspace),~1\leq k\leq \infty$. As we only work on $\fQ$ weighted spaces, we drop the second subscript from now on and set
$$P^t_\text{Smol}:=P^t_{\text{Smol},\muQ}.$$

Comparing equation (\ref{kolmogorovforwardequationsmoluchowski}) to Proposition \ref{gens}, we immediately see
\begin{corollary} \label{g2smolu}
The pseudo generator $\pgen{2}$ (of the spatial transfer operator) is the infinitesimal generator of the Smoluchowski dynamics:
$$
\pgen{2}=G_\text{Smol}.
$$
\end{corollary}

In conclusion, the density transport under spatial dynamics is similar to that under Smoluchowski dynamics, but on different timescales: The Taylor expansion of $\ops{t}$ (in  spirit of Proposition \ref{truncatedtaylor}) gives
\begin{align}
\ops{t}u &= u + t\pgen{2}u+\frac{t^2}{2}\pgen{2}^2u + \ldots~, \label{smoluexpansion}
\intertext{while that of $S^t$ gives}
S^tu & = u+ \frac{t^2}{2}\pgen{2}u +\frac{t^3}{6}\pgen{3}u + \ldots~. \label{spatialexpansion}
\end{align}
Thus formally rescaling $t\mapsto\frac{t^2}{2}$ in (\ref{smoluexpansion}) equals (\ref{spatialexpansion}) up to second order terms in $t$. We will make this rigorous in the following section.

\subsection{Local reconstruction of the spatial transfer operator}

In this section we aim to approximate $S^t$ in a way that is suitable for subsequent numerical metastability analysis. 

To ensure that the various operators constructed with the pseudo generators are well-defined, we introduce the spaces
\begin{equation}
\W{N}{K}{\Qspace}:=\big\{u\in\C{2KN}{\Qspace}~|~(\pgen{k})^nu\in \LfQ{2}\left(\Qspace\right)~\forall n=0,\ldots,N,~\forall k=0,\ldots,K\big\}.
\end{equation}

Note that, despite a similar notation, $\W{N}{K}{\Qspace}$ is not the usual Sobolev space. $2KN$ is the highest derivative appearing in~$\pgen{K}^N$, so we require the corresponding differentiability.
The choice of $\LfQ{2}(\Qspace)$ in $\W{N}{K}{\Qspace}$ is motivated by the definition of transition probabilities via the scalar product on $\LfQ{2}(\Qspace)$, \eqref{slicetransitionprob}. However, all error estimates in this section hold for $\LfQ{k}(\Qspace),~1\leq k\leq \infty$ as well, if the definition of $\V{N}{\Omega}$ is also changed accordingly.
The requirement that $(\pgen{k})^nu\in \LfQ{2}(\Qspace)$ is mostly technical in nature. Arbitrary $K$ and $N$ will only appear in Theorem \ref{spatialtaylor} and Lemma \ref{experror}. Later on, only the space $\W{1}{2}{\Qspace}$ will be of interest, due to the simple structure of $\pgen{2}$ and $\pgen{3}$. We will see later (Section \ref{sec:regularityeigenfunctions}) that $\W{1}{2}{\Qspace}$ in fact contains our objects of interest, namely the eigenvectors of our approximations to $S^t$. Moreover, it is big enough to allow for a sensible discretization basis (Section \ref{sec:spectralcollocation}).

\paragraph{Taylor reconstruction}

Combining Proposition \ref{truncatedtaylor} and Lemma \ref{altpseudogen} gives the following natural Taylor reconstruction of $S^t$:

\begin{theorem}\label{spatialtaylor}
Let $u\in \W{1}{K}{\Qspace}$. Then, 
$$
\Big\|S^tu - \sum_{k=0}^K\frac{t^k}{k!}\pgen{k}u\Big\|_{2,\muQ} = \mathcal{O}(t^{K+1}), \quad (t\rightarrow 0).
$$
\end{theorem}
\begin{proof}
By definition of $S^t$ and Lemma \ref{altpseudogen}, we can write
\begin{align*}
\Big\|S^tu &- \sum_{k=0}^K\frac{t^k}{k!}\pgen{k}u\Big\|_{2,\muQ}\\
 &= \Big\|\frac{1}{\fQ(q)}\int_\Pspace\opl{t}\big(u(q)\fcan(q,p)\big)~dp - \sum_{k=0}^K\frac{t^k}{k!} \frac{1}{\fQ(q)}\int_\Pspace \gen^{k}\big(u(q)\fcan(q,p)\big)~dp\Big\|_{2,\muQ}\\
&\leq \frac{1}{\fQ(q)}\int_\Pspace \Big\| \opl{t}\big(u(q)\fcan(q,p)\big) -  \sum_{k=0}^K\frac{t^k}{k!} \gen^{k}\big(u(q)\fcan(q,p)\big) \Big\|_{2,\muQ} ~dp. 
\end{align*}
However, the integrand is of order $\mathcal{O}(t^{K+1})$ by Proposition \ref{truncatedtaylor}.
\end{proof}

Unfortunately, the $\pgen{k}$ for $k>3$ are not readily available. In those pseudo generators, higher derivatives of the potential $V$ appear, whose analytic or numerical evaluation can be costly (they are $k$-dimensional tensors).

In practice, however, the gradient $\nabla V$ (the ''force field``) typically is available, as it would be needed for numerical simulation of the system anyway.
If we thus truncate the Taylor-like sum from Theorem \ref{spatialtaylor} after the third term, higher derivatives of $V$ are avoided, as in the computation of $\pgen{2}$ and $\pgen{3}$ only $\nabla V$ occurs. We call
\begin{equation}\label{taylorrestoredoperator}
\begin{aligned}
R^tu &:= \Big(\operatorname{id} + \frac{t^2}{2}\pgen{2} + \frac{t^3}{6}\pgen{3}\Big)u\\
&= u + \big(\frac{t^2}{2} - \gamma\frac{t^3}{6}\big)\Big(\frac{1}{\beta}\Delta u -\nabla u\cdot \nabla V\Big)
\end{aligned}
\end{equation}
the \emph{3rd order Taylor approximation}\footnote{As we never work with higher orders, we refer to $R^t$ simply as ``the Taylor approximation'' from now on.} of $S^t$. This yields the convergence result

\begin{corollary}\label{taylorrestorederror}
Let $u\in \W{1}{2}{\Qspace}$. Then
$$\big\| S^tu - R^tu\big\|_{2,\muQ} = \mathcal{O}(t^4),\quad (t\rightarrow 0).$$
\end{corollary}

\paragraph{Exponential reconstruction}

We expect that $R^t$ approximates $S^t$ well (for $t\rightarrow 0$) and can be computed cheaply, provided $\nabla V$ and $\Delta u$ are available. However, unlike $S^t$, $R^t$ is not norm-preserving and positive for densities with respect to $\fcan$, i.e. \mbox{$\|R^t u\|_{1,\muQ} \neq \|u\|_{1,\muQ}$} for $u\geq 0$. Therefore, when transporting $u$, we lose the interpretation of $\left(R^tu\right)\fcan$ as a physical density.

Moreover, for~$t$ sufficiently large, $R^tu$ is not even a contraction on $\W{1}{2}{\Qspace}$. With $\lambda \in\sigma(\pgen{2})$, ~$\lambda\neq 0$,
$$
\big|\underbrace{1+\frac{t^2}{2}\lambda-\frac{\gamma t^3}{6}\lambda}_{\in\sigma(R^t)}\big|  ~\rightarrow ~\infty, \quad (t\rightarrow\infty),
$$
and so $\|R^t\|_{2,\muQ} \rightarrow \infty, \quad (t\rightarrow\infty)$.
We will see in the numerical experiments that this quickly (i.e. already for small to moderate $t$) destroys the interpretation of the eigenvalues of $R^t$ as metastability quantifiers.

Therefore we mainly use an alternative approximation to $S^t$, called the \emph{exponential approximation} $E^t$, which is $\LfQ{1}$-norm-preserving and positive for densities, further contractive on $\W{1}{2}{\Qspace}$. One has to be careful with notation, however. As $\pgen{2}$ is an unbounded operator on $\C{2}{\Qspace}\cap \LfQ{2}(\Qspace)$, an operator exponential of form $e^{\pgen{2}}$ cannot be defined by an infinite series. However, considerations of e.g. Pazy \cite{Paz83} allow us to define $E^t$ over a bounded operator approximating $\pgen{2}$, the so-called \emph{Yosida approximation}:
$$
\pgen{2}^\lambda:=\lambda \pgen{2} (\lambda I-\pgen{2})^{-1}~~\text{for}~\lambda\in\mathbb{R}_{\geq0}~.
$$

\begin{lemma}	\label{lem:yosida}
$\pgen{2}^\lambda$ is a bounded linear operator on $\W{1}{2}{\Qspace}$, and 
$$
\lim_{\lambda\rightarrow\infty}\pgen{2}^\lambda=\pgen{2}.
$$
\end{lemma}
\begin{proof} 
As the infinitesimal generator of the Smoluchowski dynamics, $\pgen{2}$ fullfills the assumptions of \cite[Theorem 3.1]{Paz83}. Thus, the statement holds due to \mbox{\cite[Lemma 3.3]{Paz83}}.
\end{proof}

With this, we define
$$
E^tu := \exp\left(\frac{t^2}{2}\pgen{2}\right)u := \lim_{\lambda\rightarrow\infty}\exp\left(\frac{t^2}{2}\pgen{2}^\lambda\right)u,
$$
which has the desired properties:
\begin{proposition}	\label{prop:ScaledSmol}
Let $u\in\LfQ{1}(\Qspace),~u\geq 0$. Then $E^tu\geq 0$ and $\|E^tu\|_{1,\muQ} = \|u\|_{1,\muQ}.$
Moreover, $E^t$ is a contraction on $\LfQ{k}(\Qspace)$.
\end{proposition}
\begin{proof}
Due to Corollary~\ref{g2smolu} and Lemma~\ref{lem:yosida}, $E^t$ is simply a time-scaled version of the transfer operator of the Smoluchowski dynamics:
$$
E^t = \ops{{t^2}/{2}}.
$$
As such, it inherits the desired properties from $\ops{t}$.
\end{proof}

The following statements describe the approximation quality of $E^t$ for small $t$, analogous to the Taylor approximation:

\begin{lemma}\label{experror}
Let $u\in\W{N}{2}{\Qspace}$. Then for $t\rightarrow 0$,
$$
\varepsilon(t):=\bigg\|E^tu - \sum_{n=0}^N \frac{\big(\frac{t^2}{2}\pgen{2}\big)^n}{n!}u\bigg\|_{2,\muQ} = \mathcal{O}(t^{2N+1}).
$$
\end{lemma}
The proof can be found in Appendix~\ref{proofs}.

\begin{corollary}\label{exprestorederror}
Let $u\in\W{1}{2}{\Qspace}$. Then
$$
\big\|E^tu - S^tu\big\|_{2,\muQ} = \mathcal{O}(t^3) 
$$
for $t\rightarrow 0$.
\end{corollary}
\begin{proof}
\begin{align*}
\big\|E^tu - S^tu\big\|_{2,\muQ} &\leq \Big\|E^tu - \big(\sum_{n=0}^2\frac{t^n}{n!}\pgen{k}\big)u\Big\|_{2,\muQ} + \Big\|\big(\sum_{n=0}^2\frac{t^n}{n!}\pgen{n}\big)u - S^tu\Big\|_{2,\muQ}\\
&=\Big\|E^tu - \sum_{n=0}^1 \frac{\big(\frac{t^2}{2}\pgen{2}\big)^n}{n!}u\Big\|_{2,\muQ} + \Big\|\big(\sum_{n=0}^2\frac{t^n}{n!}\pgen{n}\big)u - S^tu\Big\|_{2,\muQ}.	
\end{align*}
Both summands are $\mathcal{O}(t^3)$, the first due to Lemma \ref{experror}, the second due to Theorem \ref{spatialtaylor}.
\end{proof}

\begin{remark}
\quad
\begin{enumerate}
\item An approximation order of $\mathcal{O}(t^4)$ can be achieved by including $\pgen{3}$ into $E^t$, i.e. setting
$$
E^t_3:=\exp\left(\frac{t^2}{2}\pgen{2}+\frac{t^3}{6}\pgen{3}\right) := \lim_{\lambda\to\infty}\exp\left(\left(\frac{t^2}{2}-\gamma\frac{t^3}{6}\right)\pgen{2}^{\lambda}\right).
$$
$E_3^t$ again is a time-rescaled transfer operator of the Smoluchowski dynamics. However, it is not contractive for all $t$ as $\sigma(\pgen{2})\subset(-\infty,0]$ and $\frac{t^2}{2}-\gamma\frac{t^3}{6}\rightarrow -\infty$ for $t\rightarrow \infty$. We thus stick to the lower-order approximation to retain the correct qualitative behavior for $t\rightarrow \infty$.
\item In contrast to the operator~$R^t$, which is only defined on the domain of the associated pseudo generators, the operator~$E^t$ can be defined for every~$u\in \LfQ{k}(\Qspace)$. We conjecture that Corollary~\ref{exprestorederror} holds also for this class of functions, although our proof is not extendable to this case---it uses the Taylor reconstruction to estimate the error. More advanced techniques from semigroup theory are needed, hence this will be subject of future studies.
\end{enumerate}
\end{remark}

\paragraph{Reconstruction of eigenspaces}
The error asymptotics carries over to the spectrum and eigenvectors of $S^t$, $R^t$ and $E^t$ in the following way:

\begin{corollary} \label{taylorspectralerror}
Let $u$ be an eigenvector of $R^t$ or $E^t$ to eigenvalue $\lambda$. Then $u\in\W{1}{2}{\Qspace}$ and
\begin{align*}
\big\|S^tu-\lambda u\big\|_{2,\muQ} &= \mathcal{O}(t^4)\quad \text{for } R^t,\\
\big\|S^tu-\lambda u\big\|_{2,\muQ} &= \mathcal{O}(t^3)\quad \text{for } E^t.
\end{align*}
\end{corollary}

Thus, for small~$t$ we interpret the eigenpair $(u,\lambda)$ of~$R^t$ or~$E^t$ as a good approximation to an eigenpair of~$S^t$.

\section{Numerical experiments}\label{sec:experiments}

We now show how to numerically exploit the newly-developed approximations to $S^t$ on two simple examples. We will see that the simple structure of our pseudogenerators (up to $\pgen{3}$) allows for cheap evaluation if the right discretization techniques are used. For small lag times $t$, the approximated operators can be used for accurate metastability analysis and the expected convergence rates hold.

In this setting, the discretization technique of choice are spectral collocation methods, since they are known to converge faster than any polynomial order, provided the underlying objects are (infinitely) smooth~\cite{Tre00}.

\subsection{Regularity of the eigenfunctions}\label{sec:regularityeigenfunctions}

Under sufficient regularity assumptions on the Langevin SDE, the eigenfunctions of the associated spatial transfer operator~$S^t$ and of the associated Smoluchowski dynamics~$P_{\rm Smol}^t$ (hence of~$G_2$ too) can be shown to be $\mathcal{C}^{\infty}$, i.e.\ smooth. This allows for an extremely efficient approximation by spectral methods. All one needs is the following assumption.
\begin{assumption}	\label{assu:smoothness}
Let the potential $V:\mathbb{R}^d\to\mathbb{R}$ be $\mathcal{C}^{\infty}$ and let all its derivatives of order greater or equal two be bounded.
\end{assumption}
The proof of smoothness touches upon techniques beyond the scope of this manuscript, hence we deferred it to Appendix~\ref{sec:smoothness}.

\subsection{Spectral collocation}\label{sec:spectralcollocation}

All our operators will be discretized and compared using spectral collocation methods. We refer to \cite{Fro13}, where these methods have been applied to transfer operator semigroups, i.e.\ the case where a ``true'' infinitesimal generator exists.  Note that in contrast to \cite{Web12}, we do not need to compute boundary integrals on high dimensional domains.

To avoid dealing with boundary conditions, we restrict ourselves to periodic position spaces. For a more general introduction see, for example,~\cite{Tre00}.

For $\Qspace=\mathbb{T}^1$, the $1$-dimensional unit circle, define for $n$ odd the finite dimensional approximation space $\mathcal{U}_n$ of trigonometric polynomials with basis of Fourier modes
\begin{equation}\label{fourierbasis}
\big\{\phi_k\big\}_{-\frac{n-1}{2}\leq k\leq \frac{n-1}{2}}~,\quad \phi_k(q)=e^{2i\pi kq},
\end{equation} 
as well as the corresponding collocation nodes $\Qspace_n:=\{0,1/n,\ldots,(n-1)/n\}$.\\
with fitting collocation nodes $\Qspace_n:=\big\{\cos\big((2k-1)\pi/(2n)\big), k=0,\ldots,n-1\big\}$.\\
Multi-dimensional phase spaces, consisting of cartesian products of $\mathbb{T}^1$, can be discretized using a corresponding product basis.

Let $I_n:\mathcal{U}\rightarrow \mathcal{U}_n$ be the corresponding projection from some function space $\mathcal{U}$ (in our case $\LfQ{2}(\Qspace)$) to $\mathcal{U}_n$. For an operator $A:\mathcal{U}\rightarrow\mathcal{U}$, we then define the discretized operator $A_n:\mathcal{U}_n\rightarrow\mathcal{U}_n$ by
\begin{equation}\label{discretizedoperator}
A_n:\mathcal{U}_n\rightarrow\mathcal{U}_n,\quad A_nf:=I_nAf.
\end{equation}
The operator $A_n$ has a matrix representation which, for the sake of notational simplicity, is also denoted by $A_n$:
$$
A_n= \big(A\phi_i(q_j)\big)_{ij}, 
$$

We are interested in parts of the spectrum and the associated eigenspaces of $A$, for $A=S^t,R^t,E^t$ or $\pgen{2}$. So instead of solving $Av=\lambda v$ on $\mathcal{U}$, we solve $A_nv = v$ on $\mathcal{U}_n$. In matrix form, using the collocation matrix $M_n=\big(\phi_i(q_j)\big)_{ij}$, this becomes a generalized eigenvalue problem on $\mathbb{C}^n$:
\begin{equation}
A_n c = \lambda M_nc \quad \text{or} \quad (A_n-\lambda M_n) c = 0,
\end{equation}
where $\sum_{k=1}^nc_k\phi_k$ then is the approximative eigenfunction of $A$ to eigenvalue $\lambda$ in Laplace space.

Under mild smoothness conditions on the potential $V$, readily fulfilled if Assumption \ref{assu:smoothness} is satisfied, the Fourier and Chebyshev bases fulfill the requirements of Corollaries \ref{taylorrestorederror} and \ref{exprestorederror}, i.e. $\phi_k\in\W{1}{2}{\Qspace},~k=1,\ldots,n$.

\subsection{Ulam's method}

Since even for our simple academic examples, $S^t$ cannot be computed analytically, a reference method is needed with which we can compute the eigenvalues and -vectors of $S^t$. For this, we here employ Ulam's method (see e.g. \cite{Del01}).

Let $\{A_1,\ldots,A_n\}$ be a partition of $\Qspace$ into subsets of positive Lebesgue measure\footnote{Typical choices are (uniform) hyperrectangles or Voronoi cells.}. Define the approximation space $\mathcal{U}_n=\operatorname{span}\left(\chi_{A_1},\ldots,\chi_{A_n}\right)$. Now, the discretization of $S^t$, denoted by $S^t_n:\mathcal{U}_n\rightarrow\mathcal{U}_n$, is the Galerkin projection of $S^t$ onto $\mathcal{U}_n$. It has the matrix representation
\begin{align*}
\left(S^t_n\right)_{i,j} &= \frac{1}{\mu(A_j)}\int_{A_j}S^t\chi_{A_i}~d\muQ\\
&= \frac{1}{\mu(A_j)}\int_{\Gamma(A_i)} \fcan(q,p) p\big(t,(q,p),\Gamma(A_j)\big)~d(q,p),
\end{align*}
with $p(\cdot,\cdot,\cdot)$ the stochastic transition function of Langevin dynamics (\ref{transitionprobability}).

This integral can be computed numerically by Monte Carlo quadrature, which involves sampling~$\Gamma(A_i)$, numerically integrating the Langevin equations for time $t$, and counting transitions to~$\Gamma(A_j)$. 

%

\subsection{Example: double well potential}

To accurately test the approximation quality of $R^t$ and $E^t$ to $S^t$, we first analyze a simple one-dimensional Langevin system on the unit circle, which can be discretized to high resolution. It has the periodic potential
$$
V(q) = 1 + 3\cos(2\pi q) + 3\cos^2(2\pi q) - \cos^3(2\pi q).
$$

\begin{figure}[h]
\centering
\includegraphics{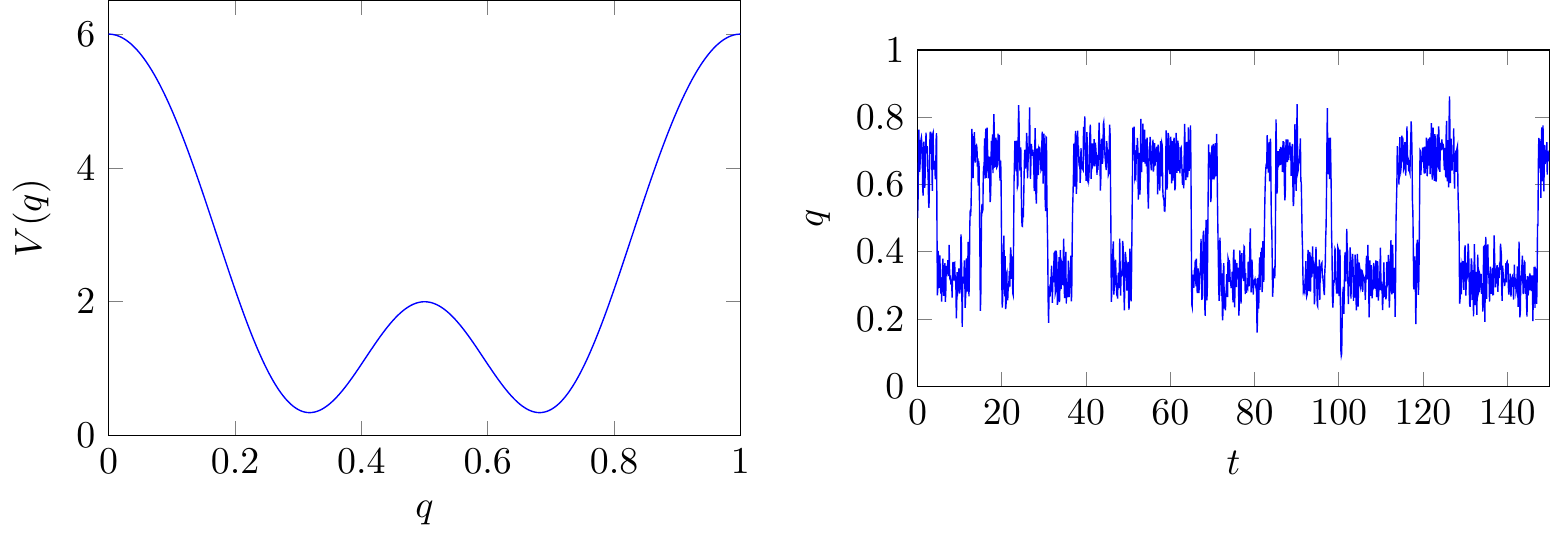}
\caption{The two wells of the periodic double well potential indicate two metastable regions in configuration space. A trajectory of the according Langevin dynamics with appropriate temperature shows the characteristic jumping pattern between wells.}
\end{figure}

We consider the system at inverse temperature $\beta=1$ and Langevin damping constant $\gamma=1$. As even for this simple system neither eigenvalues nor -vectors can be computed analytically, we first compute a classical Ulam approximation to $S^t$ with a large number of discretization boxes $N$ and sampling points $M$, denoted by $S^t_N$. Spectra and eigenvectors of $S^t_N$ then serve as a reference point for the error analysis. A resolution of $N=2^{10}$ and $M=10^4$ sampling points produce sufficiently accurate spectral data, as a further increase does not alter the results considerably.

As our goal is metastability analysis, we analyze the error in the portion of the spectrum and the eigenvectors that are necessary to identify almost invariant sets. Table \ref{tab:eigenvaluesdoublewell} shows a distinct \emph{spectral gap} after the second eigenvalue, so analysing $\lambda_1,\lambda_2$ and $v_1, v_2$ should reveal the principal metastable sets.

\begin{figure}[H]
	\centering
	
	\begin{tabular}{r | l l l l l l l}
	EV \# & 1& 2& 3& 4& 5\\
	\hline
	$t=0.1$ & 1.0000 & 0.9428 & 0.4324 & 0.3139 & 0.2022 \\
	$t=1$ & 1.0000 & 0.6620 & 0.1775 & 0.0515 & 0.0401 
	\end{tabular}
	\caption{The five largest eigenvalues of $S^t_N$ for different lag times $t$. Note the spectral gap after $\lambda_2$.}
	\label{tab:eigenvaluesdoublewell}
\end{figure}

For the discretized approximative operators $\pgen{2,n}$, $R^t_n$ and $E^t_n$, we use $n=33$ approximation functions of form (\ref{fourierbasis}) and the same number of collocation points. The fourier modes have inherent periodic boundary conditions.

\paragraph{Eigenvalue comparison}

The absolute error in the relevant eigenvalues can be measured by
\begin{align*}
	\varepsilon_R(t) &:= \big|\lambda_1(S^t_N)-\lambda_1(R^t_n)\big| + \big|\lambda_2(S^t_N)-\lambda_2(R^t_n)\big|\\
	\varepsilon_E(t) &:= \big|\lambda_1(S^t_N)-\lambda_1(E^t_n)\big| + \big|\lambda_2(S^t_N)-\lambda_2(E^t_n)\big|,
\end{align*}
where $\lambda_k(S^t_N),\lambda_k(R^t_n),\lambda_k(E^t_n)$ is the $k$-the eigenvalue of $S_N,R_n,E_n$, respectively.

\begin{figure}[H]
\centering
\includegraphics{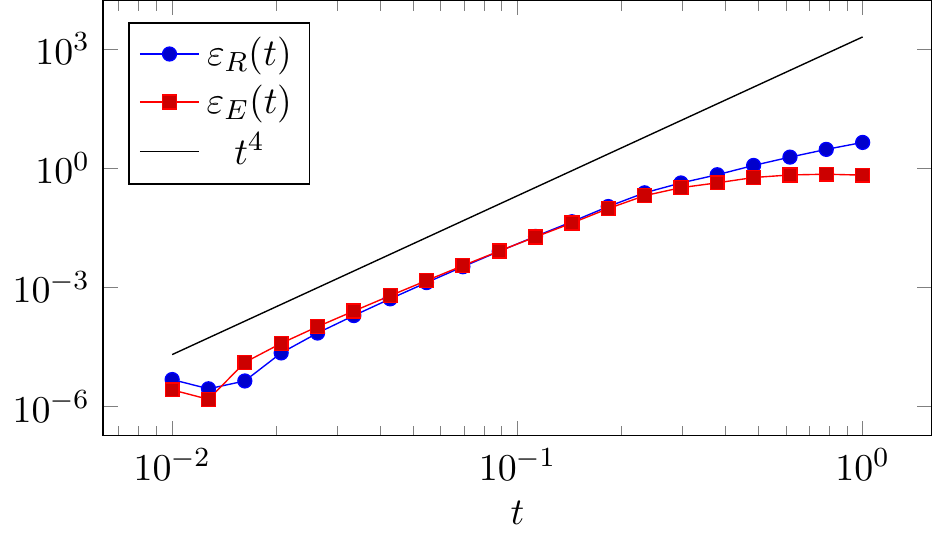}
		\caption{Eigenvalue errors for the Taylor and Exponential approximation for small $t$. $\varepsilon_R(t)$ is consistent with the estimated convergence rate $\mathcal{O}(t^4)$. $\varepsilon_E(t)$ even (visually) exceeds the predicted convergence rate of $\mathcal{O}(t^3)$.}
		\label{1derror}
\end{figure}

In Figure \ref{doublewellevals} the 8 largest eigenvalues for increasing lag times are shown. We see that for small $t$, a decent approximation of the eigenvalues of $S^t_N$ can be expected from both $R^t_n$ and $E^t_n$. However, for bigger $t$, the third-order Taylor approximation $R^t$ becomes worse, while $E^t$ at least shows the proper qualitative behavior. Added for comparison, the Smoluchowski spectrum, representing a completely different dynamics, does not resemble the spectrum of $S^t_N$.

\begin{figure}[H]
\centering
\includegraphics{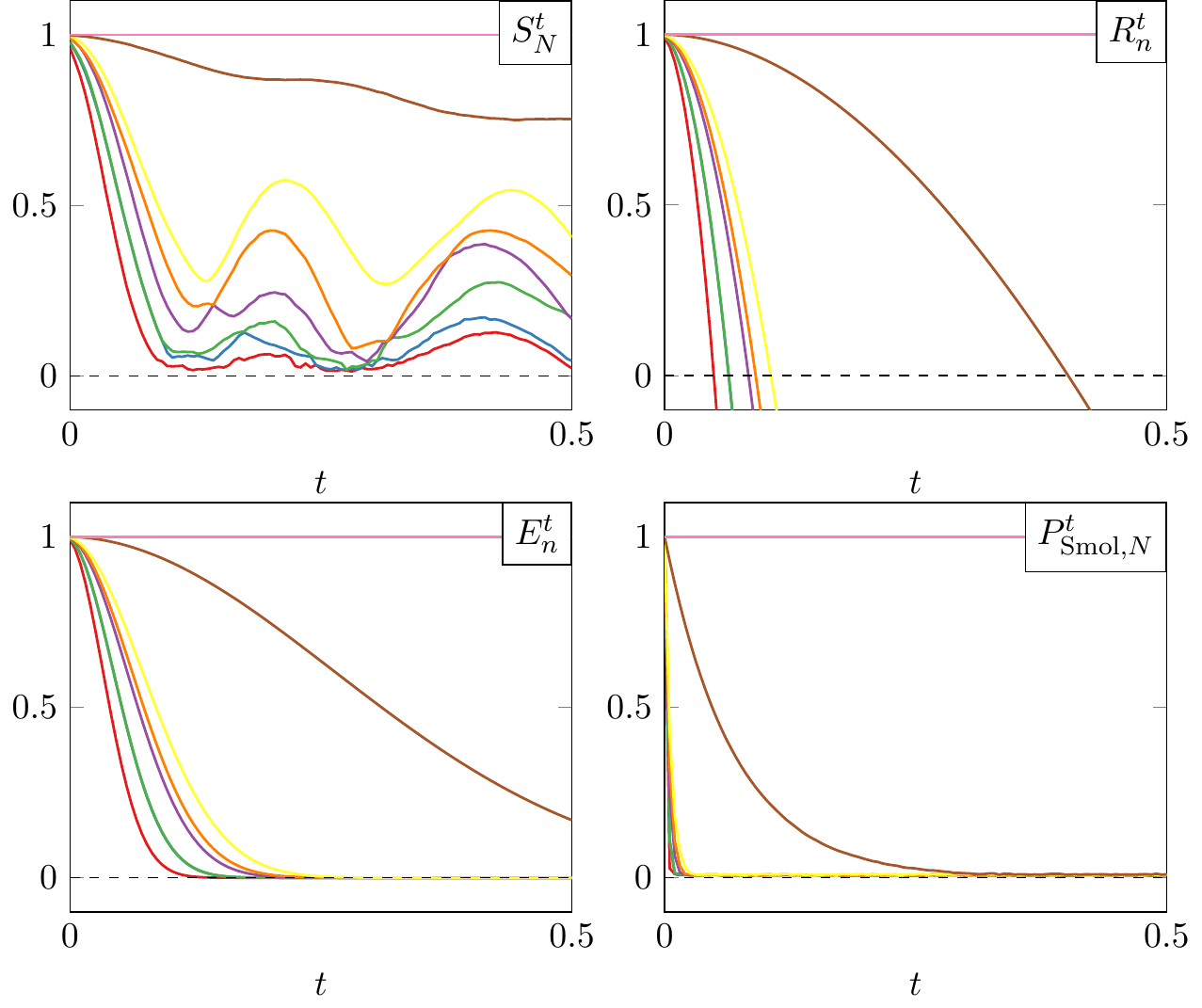}
	\caption{The solid lines show the course of the 8 biggest eigenvalues of the various discretized operators. For small lag times, both the spectra of $R^t$ and $E^t$ show the right asymptotics, while for $t\rightarrow \infty$, $E^t$ at least has the qualitative behavior. The spectrum of the Smoluchowski transfer operator compares poorly to that of $S^t_N$.}
	\label{doublewellevals}
\end{figure}

\paragraph{Eigenvector comparison}

Before comparing to the approximated operators, observe that the eigenvectors of $S^t$ (and $S^t_N$) are time-dependent. This is contrastive to any semi-group transfer operator, whose eigenvectors coincide with those of its infinitesimal generator for all times. While the first eigenvector remains constant $v_1\equiv 1$, the second eigenvector starts out as almost a step function, but gets more concentrated in the potential wells for increased lag times (Figure \ref{ev_timedependence}).

The eigenvectors of $R^t$ and $E^t$, however, are time-invariant, and coincide with those of $\pgen{2}$, by construction. For small lag times, the second eigenvector $w_2$ of $\pgen{2,n}$ (and thus $R^t_n$ and $E^t_n$) compares well to the second eigenvector $v_2$ of $S^t_N$ (see Figure \ref{1derror}). For larger $t$, when the $v_2$ becomes more and more concentrated at the potential wells, the difference is more noticable.

\begin{figure}[H]
\centering
\includegraphics{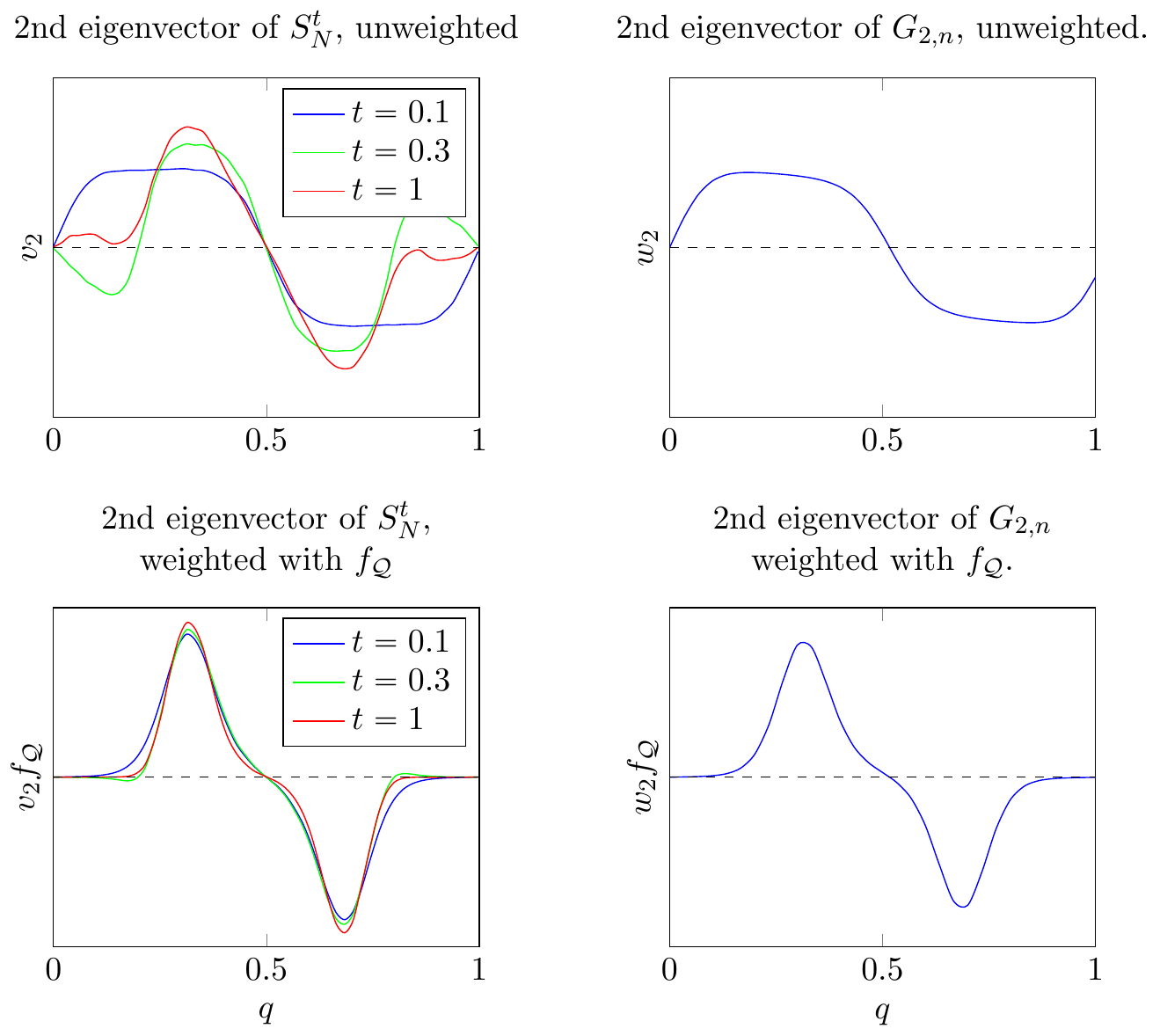}
	\caption{Visual comparison of the second largest eigenvectors of $S^t_N$ for short and intermediate lag times $t$ (left column), and of $\pgen{2,n}$ (right column). For $S^t_N$, we see strong time-dependency in the unweighted case.}
	\label{ev_timedependence}
\end{figure}

As we consider all spatial operators on $\LfQ{2}(\Qspace)$, their physical interpretation is to transport densities with respect to $\fQ$ (see also the next paragraph). It is therefore appropriate to weight their eigenvectors with $\fQ$, as this gives their representation with respect to the Lebesgue measure. In the weighted space, the time-dependence of $v_2$ becomes insignificant, as can be seen in Figure \ref{ev_timedependence}. Consequently, our restoration provides a very good approximation (Figure~\ref{ev_timedependence}). 

The sign structure of $w_2$ now identifies the pair of metastable sets 
$$A_1=\{w_2>0\}=(0,0.5),\quad A_2=\{w_2<0\}=(0.5,1).$$

\paragraph{Transition probabilities}

Theorem \ref{metastability} also provides estimates for the transition probabilities between those sets: the combined ``degree of metastability'' $p(t,A_1,A_1)+p(t,A_2,A_2)$ can be bounded from above and below by functions of the second largest eigenvalue $\lambda_2(S^t)$. To verify this numerically, and to examine the approximation quality of bounds based on the eigenvalue $\lambda_2(E^t)$\footnote{The approximation of the bounds based on $R^t$ is omitted here, as the only difference to $E^t$ is the (already demonstrated) rate of decay for $t\rightarrow \infty$.} , we first estimate $p(t,A_i,A_i)$ by Monte Carlo integration:
\begin{enumerate}
\item Sample two sets of starting points, with density $\chi_{A_1} \fQ$ and $\chi_{A_2} \fQ$.
\item Integrate those samples numerically for some fixed time $t$ under the Langevin dynamics.
\item Count the portion of points that remained in $A_1$ and $A_2$, respectively.
\end{enumerate}
For sufficiently many samples, this provides an accurate estimate for $p(t,A_1,A_1)$ and $p(t,A_2,A_2)$.

Figure \ref{metastabilitybounds} (left) confirms the bounds based on $\lambda_2(S^t)$. They do, however, provide a relatively large margin of error, and diverge for increasing lag time\footnote{It can be argued that weighting the initial distribution of points by $|w_2|$ gives a ``more physical'' portion of the canonical density, i.e. $\chi_{A_i} |w_2| \fQ$. The corresponding degree of metastability does, in fact, coincide much better with the upper bound.}.
Figure \ref{metastabilitybounds} (right) now shows that \emph{for small lag times}, the bounds based on $\lambda_2(E^t)$ also contain $p(t,A_1,A_1)+p(t,A_2,A_2)$, and thus are an accurate approximation to the bounds based on $\lambda_2(S^t)$ in this time region.

\begin{figure}[H]
\centering
	\includegraphics{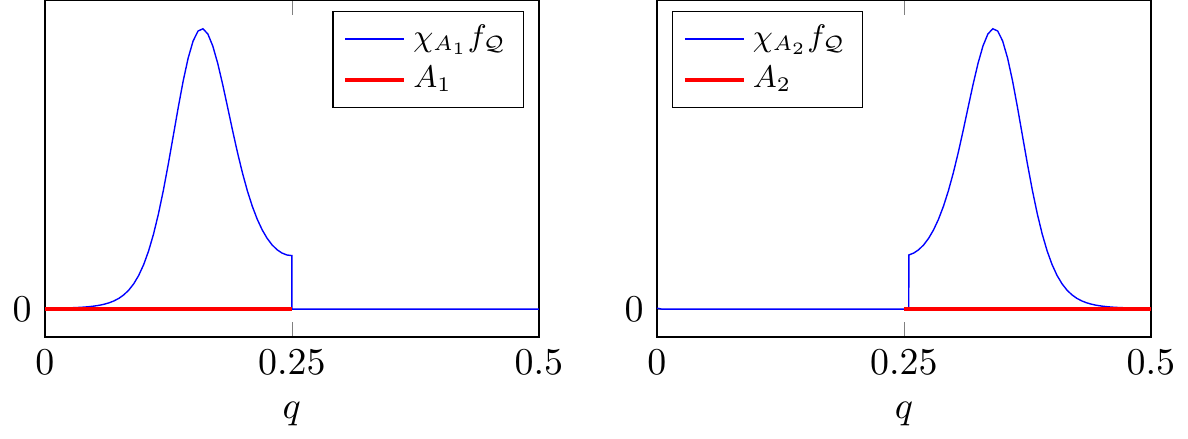}
	\caption{The invariant sets with the corresponding portion of the canonical density.}
\label{ev_sets}
\end{figure}

\begin{figure}[H]
\centering
	\includegraphics{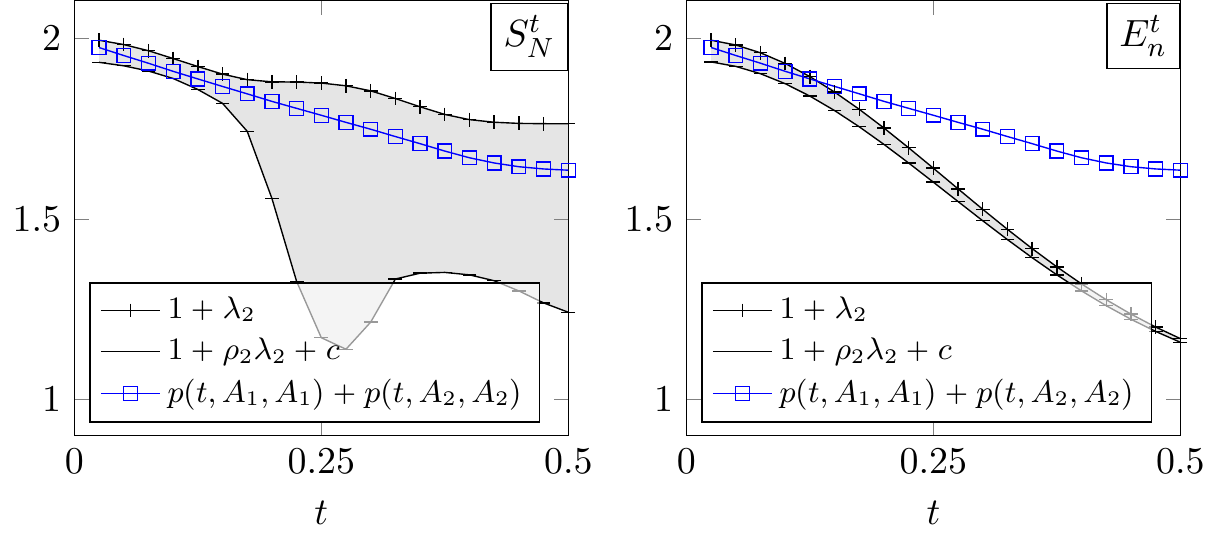}
	\caption{Metastability of the partition $A_1, A_2$ and comparison to the bounds of Theorem \ref{metastability}, calculated based on the second largest eigenvalue of $S^t$ (left) and $E^t$ (right).}
    \label{metastabilitybounds}
\end{figure}

A recent continuation of the pseudo generator theory~\cite{BiHaKoJu15} shows that it may be possible to extend the time scales for which meaningful information can be extracted from the reconstructed operator, i.e. for which $\lambda_2(E^t)$ resembles $\lambda_2(S^t)$. It has been discussed that $S^t$ exhibits a kind of "almost-Markovianity", which results in a basically exponential decay of $\lambda_2(S^t)$ for high enough damping $\gamma$ and large enough $t$. Transferring this decay rate to $\lambda_2(E^t)$, it has been demonstrated that $\lambda_2(S^{k\cdot \tau})\approx \big(\lambda_2(E^\tau)\big)^k$, for the right choice of $\tau$ and $k$ large enough.

\subsection{Example: four well potential}

The discretization and restoration method performs similarly on higher-dimensional domains. Using grid-based spectral collocation, we can reconstruct the spatial transfer operator for a two-dimensional four well potential of the form
\begin{equation}
\begin{aligned}[3]
V(q_1,q_2) &= 1 + 3\cos(2\pi q_1) + 3\cos^2(2\pi q_1) - \cos^3(2\pi q_1) \\
&~~ + 1 + 3\cos(2\pi q_2) + 3\cos^2(2\pi q_2) - \cos^3(2\pi q_2) + \cos(2\pi q_2-\frac{\pi}{3}).
\end{aligned}
\end{equation}

\begin{figure}[H]
	\centering	
	\includegraphics[width=8cm,keepaspectratio]{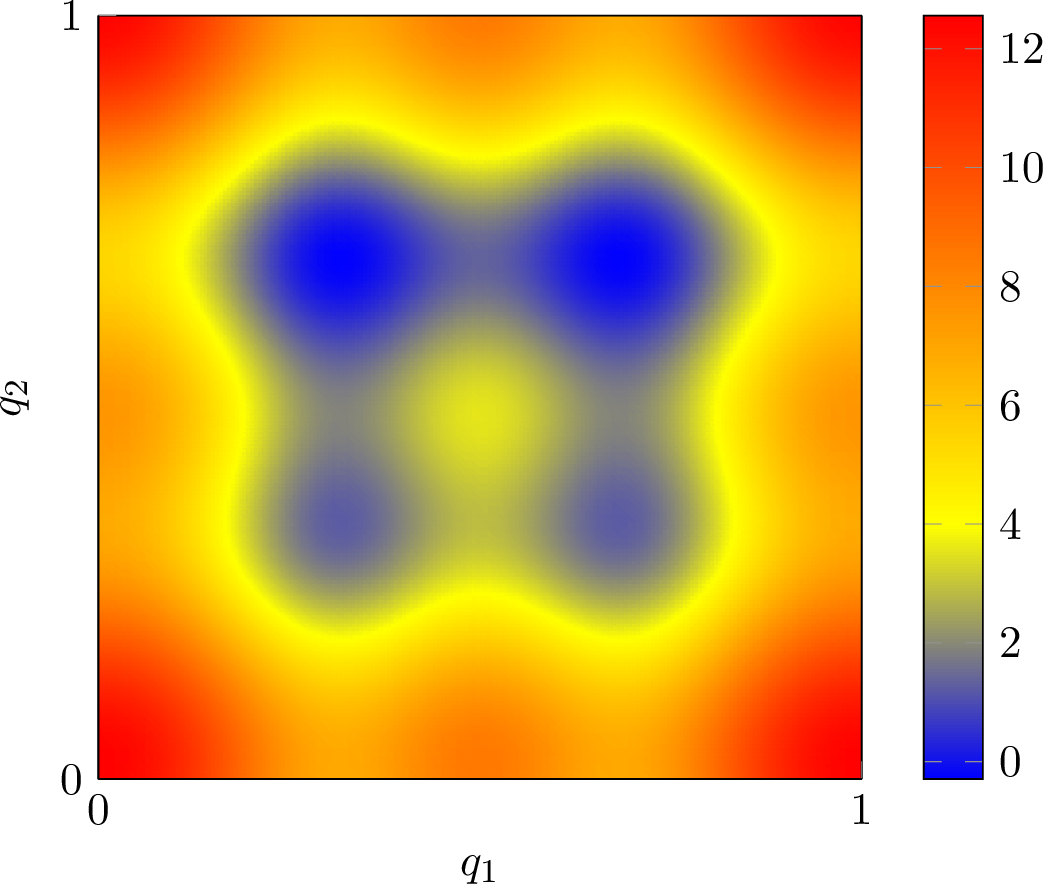}
\caption{The four well potential in the region $[0,1]^2$.}
\end{figure}

This potential is of interest, as the four local minima of different depth form multiple hierarchies of metastable sets. A similar (albeit non-periodic) potential was considered in \cite{Deu96}. Again, the potential is periodic on $\mathbb{T}^2$, so for discretization we use (products of) Fourier modes. As for this example we do not perform rigorous error analysis, a resolution of 33 basis functions and collocation points per dimension is sufficient for both the Ulam and collocation discretizations, resulting in a total of 961 basis functions. We use heat and damping parameters $\beta=1,~\gamma=1$.

The spectrum of the (Ulam-approximated) spatial transfer operator shows a significant gap after the fourth eigenvalue; we thus expect to identify four metastable sets, corresponding to the four potential wells.

\begin{figure}[H]
	\centering
	
	\begin{tabular}{r | l l l l l l l}
	EV \# & 1& 2& 3& 4& 5& 6\\
	\hline
	$t=0.1$ & 0.9974 & 0.9053 & 0.8950 &0.8122 & 0.4063 & 0.3647 \\
	$t=1$ & 0.9873 & 0.7518 & 0.7307 & 0.5569 & 0.2894 & 0.2769 
	\end{tabular}
	\caption{The six largest eigenvalues of $S^t_n$ for short and intermediate lag times $t$. Note the spectral gap after $\lambda_4$.}
\end{figure}

For the significant eigenvectors $v_2,v_3,v_4$ of $S^t_n$ and their approximations $w_2,w_3,w_4$ via $\pgen{2,n}$, we observe a similar behavior as in the one-dimensional case: For longer lag times, the relevant eigenvectors get more and more concentrated in the regions of the potential wells. Again, the sign structure is largely identical (Figure \ref{fig:fourwelleigenvectors}).

\begin{figure}[h]
\centering
	\includegraphics{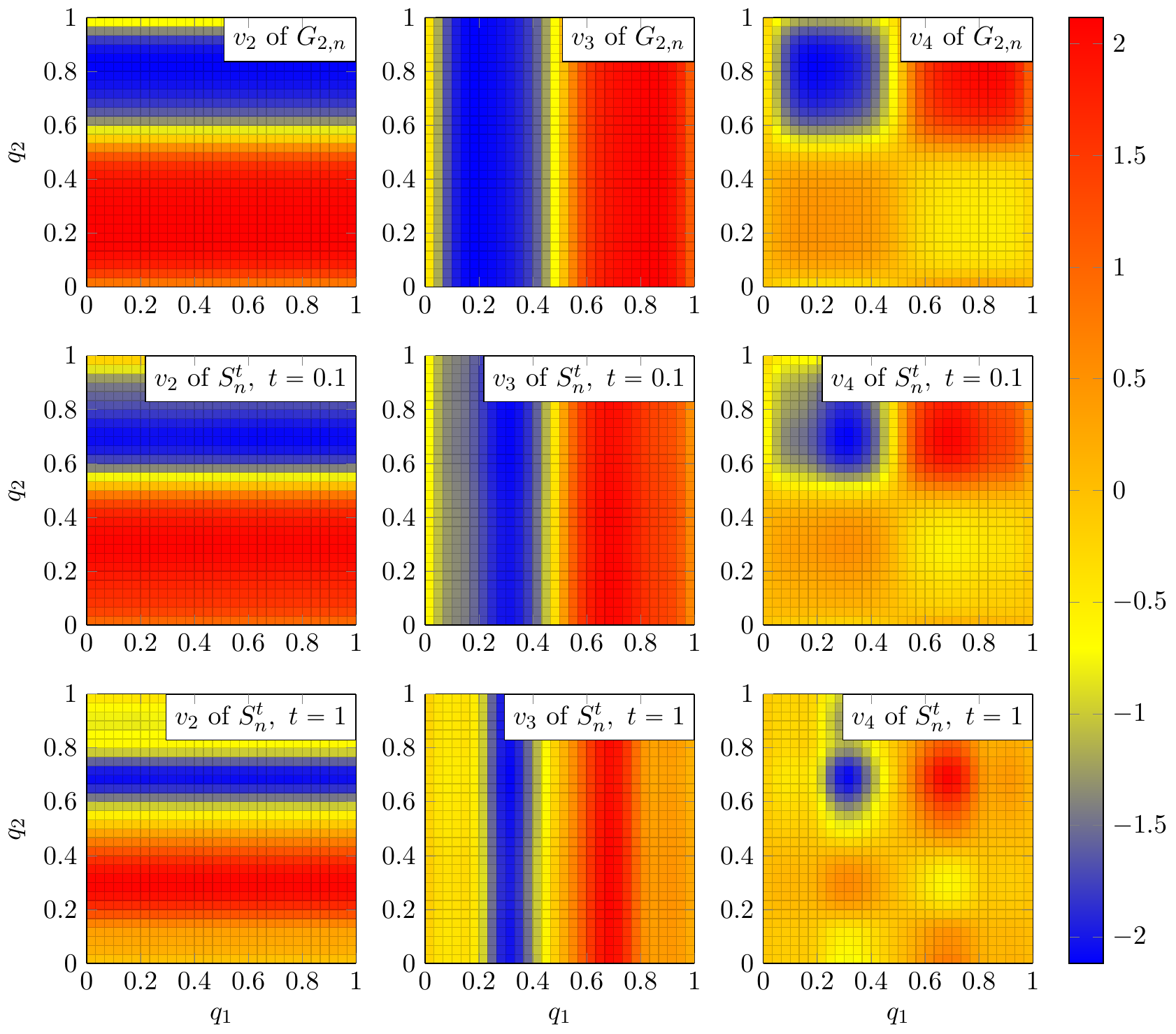}
\caption{The three most significant eigenvectors of $\pgen{2,n}$ (top row) and $S^t_n$ for different lag times (second and third row), unweighted.}
	\label{fig:fourwelleigenvectors}
\end{figure}

\paragraph{Transition probabilities and hierarchies of metastability}

The sign structure of the three isolated eigenvectors partitions $\Qspace$ into three pairs of metastable sets, each with a different ``degree of metastability'' (Figure \ref{fourwellsets}). The portion of $\fQ$ resembling the respective metastable sets now is
$$    f_{A_i} = \chi_{A_i}  \fQ,\qquad
    f_{B_i} = \chi_{B_i}  \fQ,\qquad
    f_{C_i} = \chi_{C_i}  \fQ,\qquad i=1,2.    
$$
Again sampling these densities, integrating and counting the points remaining in the respective sets shows the connection to the dominant eigenvalues of $S^t$ as of Theorem \ref{metastability}. Again, for short times, the bounds based on the spectrum of $E^t_n$ captures the decay of metastability quite well (Figure \ref{metastabilityboundsfourwell}).

\begin{figure}[H]
\centering
	\includegraphics{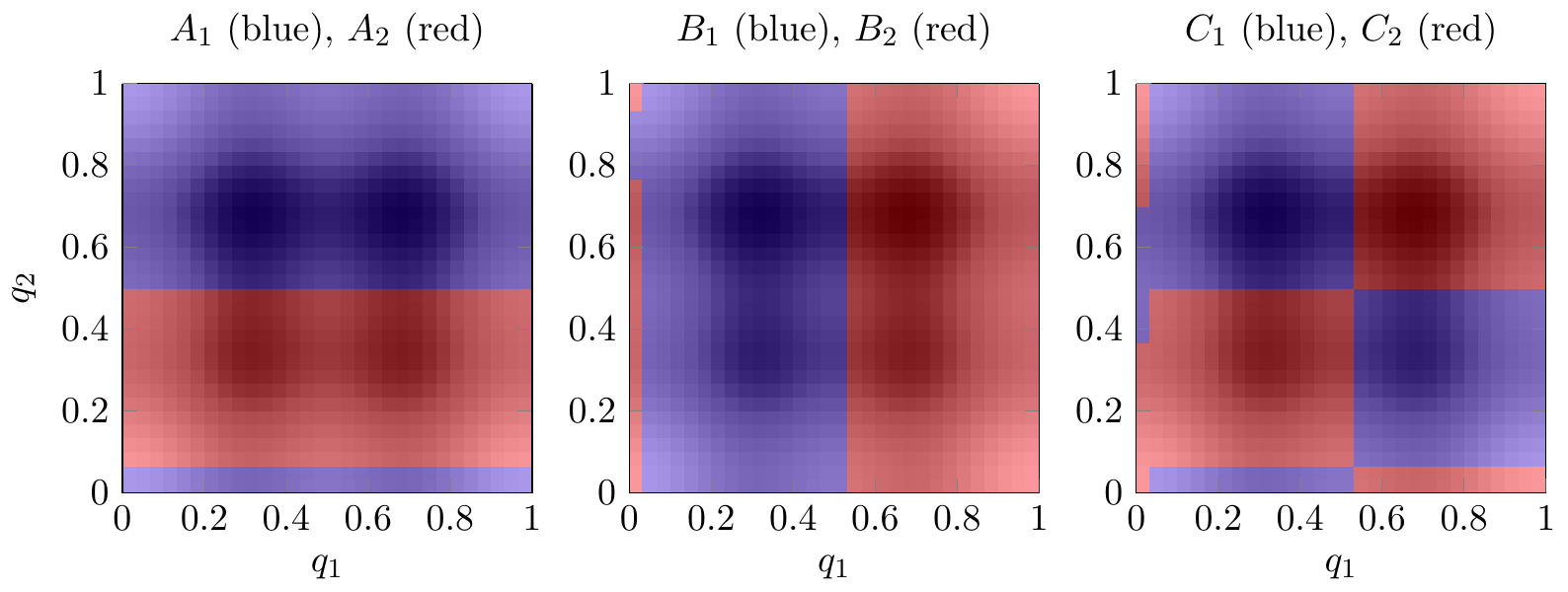}
\caption{The three tiers of invariant sets, identified via $\pgen{2,n}$. The artifacts on the left border of $C_1, C_2$ can be attributed to the ill-conditioned sign structure analysis.}
\label{fourwellsets}
\end{figure}

\begin{figure}[H]
\centering
	\includegraphics{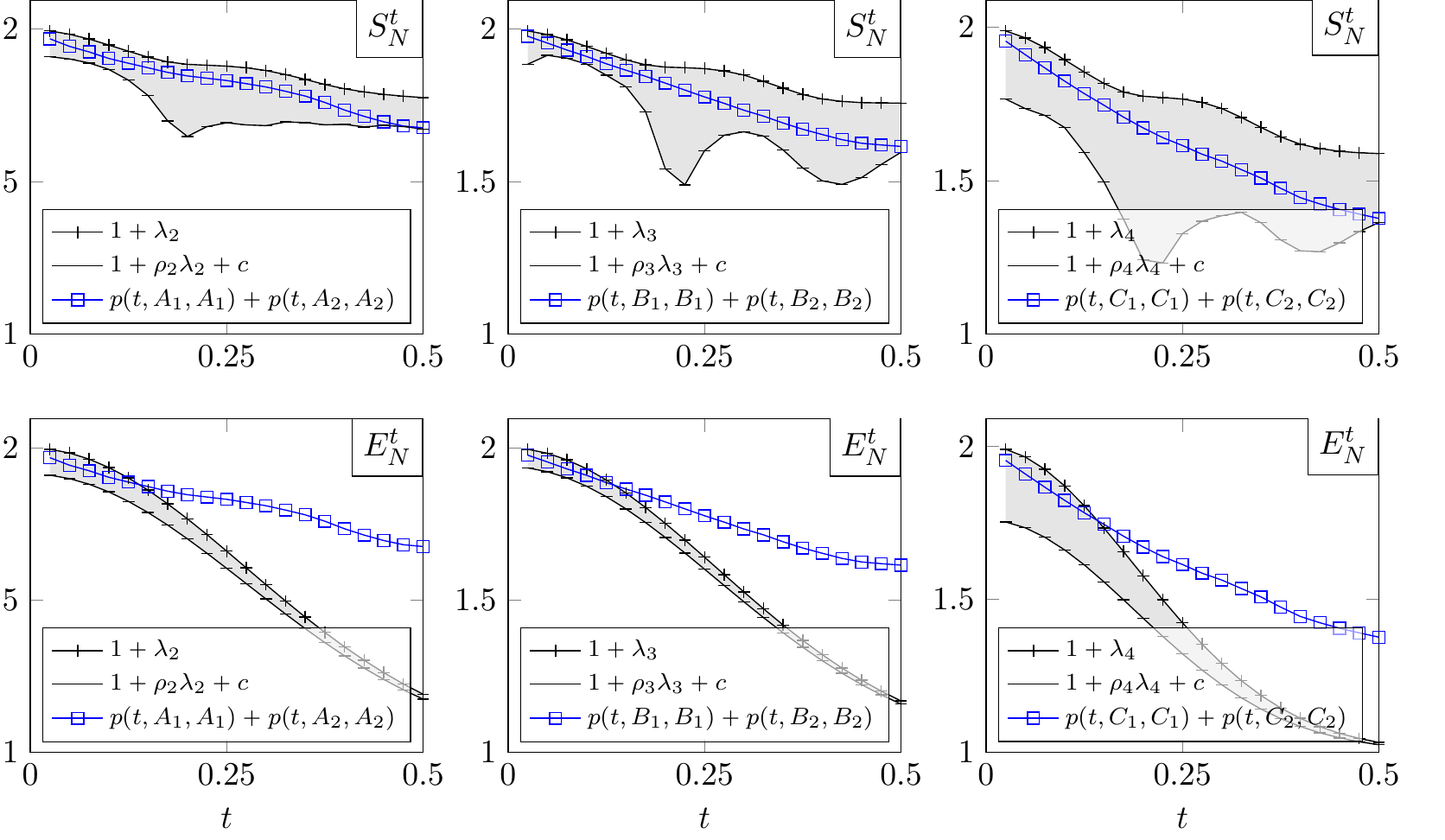}
	\caption{Combined metastability of the identified sets and comparison to the bounds of Theorem \ref{metastability} for both $S^t$ and $E^t$.}
    \label{metastabilityboundsfourwell}
\end{figure}

\section{Conclusion and future work}

We have considered the dynamics of the position coordinate for a molecular dynamics system given by the Langevin process in thermal equilibrium. Following the aim of an efficient, trajectory-free evaluation of the associated \emph{spatial transfer operator}, we have found that the spatial dynamics behaves up to third order in time as a $t\mapsto \tfrac{t^2}{2}$ scaled Smoluchowski dynamics (for $t\to 0$); cf.\ Proposition~\ref{prop:ScaledSmol} and Corollary~\ref{exprestorederror}.

The numerical experiments suggest that our theoretical findings on the asymptotic approximation error can be extended to the dominant spectrum as well, hence that the approach is applicable for metastability analysis.  In order to be applicable to bio-chemically relevant systems, two main points have to be addressed: (a) extension of the approximation quality for larger time scales, and (b) dealing with the numerical approximation for higher dimensional systems. Thus we are going to investigate the following topics in future studies:

$\bullet$ Corollary~\ref{exprestorederror} tells us that a time-scaled Smoluchowski dynamics approximates the spatial dynamics of a Langevin process up to third order in time -- independently of the Langevin damping coefficient~$\gamma$. Of course, a quantitative estimate would depend on the damping, and it is important to understand how the approximation quality behaves for varying damping and larger lag times.

$\bullet$ It is shown in Appendix~\ref{sec:spatial_spectrum} that the spatial dynamics is ergodic, which suggests an exponential decay of the eigenvalues of the spatial transfer operator as $t\to\infty$. Incorporating this qualitative property into the structure of the reconstruction (i.e.\ the approximation of the spatial transfer operator from the pseudo generators) is key to be able to leap to larger lag times and have the time scales of the original system approximated properly.

$\bullet$ The potential, governing the self-dynamics of a molecule, can be formulated in terms of \emph{internal coordinates} of the molecule: bond lengths, bond angles, and dihedral angles. Describing the dynamics of the system in these coordinates maintains its essential dynamical properties, but reduces its dimension. We shall examine how the pseudo generator approach can be transferred into these coordinates, and whether the dimensional reduction pays off against the imposed additional difficulties (e.g.\ the mass matrix, which is a constant diagonal matrix in Cartesian coordinates---taken to be the identity in~\eqref{hamiltondynamics} and \eqref{langevindynamics}---, is in general a position-dependent full matrix in the internal coordinates~\cite{Fri09}).

$\bullet$ The internal coordinates above often play an even more prominent role: in many cases, the dominant conformational transitions can be described by only a few of them; e.g.\ the $\alpha \rightleftharpoons \beta$ transition for alanin dipeptide occurs in the $\phi/\psi$ coordinate plane. Evidently, if a faithful projection of the full Langevin dynamics to these \emph{essential coordinates} can be carried out, this yields a massive dimensionality reduction. Certainly, however, through this projection (which is additional to the already imposed averaging over the momenta) dynamical information is lost, and we expect the approximation quality by the associated pseudo generators to deteriorate.

$\bullet$ Especially in the latter case, the use of higher order pseudo generators could be advantageous. It shall be investigated how to incorporate them into the reconstruction, and assessed whether the increase in accuracy due to increased approximation order is worth the effort to compute higher order derivatives of the potential---as they will probably enter the expressions. However, we have seen in Proposition~\ref{gens} that $\pgen{3} = -\gamma \pgen{2}$, hence no additional potential evaluations are needed. Whether this is a generic pattern for pseudo generators or a singular fluke will be examined in future studies.

$\bullet$ Cases will appear where despite all dimension and model reduction techniques we have to deal with the discretization of operators on function spaces over medium- to high-dimensional spatial domains, which are subject to the curse of dimension.  Following \cite{Web12}, we intend to use meshfree approximation methods where the nodes are distributed according to prior knowledge gained from trajectory data.  The faith in this approach resides in recent results of \emph{transition path theory}~\cite{EVE04,EVE06,EVE10}, which shows that the vast majority of conformational transitions occur along a few dominant, low dimensional transition pathways.

\section*{Acknowledgments}

The authors would like to thank Carsten Hartmann for helpful discussions.

\appendix

\section{Derivation of pseudo generators by vector calculus} \label{proofs}

\begin{proof}[Proof of Proposition~\ref{gens}]

By Lemma \ref{altpseudogen}, we obtain $\pgen{n}$ by calculating $\genl^{n},~n\in\{1,2,3\}$ and applying the momentum integral afterwards. For readability, we use the shorthand~$L$ instead of~$\genl$ for the remainder of this proof. Let $u\in \LfQ{2}(\Qspace)\cap\mathcal{C}^4(\Qspace)$.

1.
$\gen$ is the differential operator of the Fokker--Planck equation (\ref{fokkerplanckequationlangevin}), and therefore
\begin{align*}
\gen\big(u(q)\fcan(q,p)\big)&=\left(\frac{\gamma}{\beta}\Delta_p - p\cdot\nabla_q+\nabla_qV\cdot\nabla_p+\gamma p\cdot\nabla_p+d\gamma\right)\big(u(q)\fcan(q,p)\big)\\
&=\frac{\gamma}{\beta}u(q)\Delta_p\fcan(q,p)\\
&\qquad-p\cdot\big(\nabla_qu(q)f(q,p)+u(q)\nabla_q\fcan(q,p)\big)\\[1ex]
&\qquad+\nabla_qV(q)\cdot\big(u(q)\nabla_p\fcan(q,p)\big)\\[1ex]
&\qquad+\gamma u(q) p\cdot \nabla_p\fcan(q,p)\\[1ex]
&\qquad+ d\gamma u(q)\fcan(q,p).
\intertext{Setting $\fcan(q,p)=\efcan$, this becomes}
	&=\frac{1}{Z}e^{-\beta\left(\frac{p\cdot p}{2}+V(q)\right)}\Big[\frac{\gamma}{\beta}(\beta^2p\cdot p-d\beta)u(q) - \left(p\cdot\nabla_qu(q)-\beta p\cdot\nabla_qV(q)u(q)\right)\\
	&\qquad -\beta p\cdot\nabla_qV(q)u(q) -\beta\gamma p\cdot pu(q) + d\gamma u(q)\Big]\\
	&=-\frac{1}{Z}e^{-\beta\left(\frac{p\cdot p}{2}+V(q)\right)} p\cdot \nabla_qu(q).
\end{align*}
Applying the integral (and normalizing) in (\ref{spatialto}), we get
\begin{align*}
\pgen{1}u(q) &=  \frac{1}{\fQ(q)}\int_\Pspace\partial_tT^t\Big|_{t=0}\big(u(q)\fcan(q,p)\big)dp\\
	&=  \frac{1}{\fQ(q)}\int_\Pspace-\frac{1}{Z}e^{-\beta\left(\frac{p\cdot p}{2}+V(q)\right)} p\cdot \nabla_qu(q)dp\\
	&=-\frac{1}{\fQ(q)}\frac{1}{Z}e^{-\beta V(q)}\left[\int_\Pspace e^{-\beta\frac{p\cdot p}{2}}\left(p\cdot\nabla_qu(q)\right)dp\right]\\
	&=-\frac{1}{\fQ(q)}\frac{1}{Z}e^{-\beta V(q)}\sum_{k=1}^{d}\left[\partial_{q_k}u(q)\int_\Pspace p_ke^{-\beta\frac{p\cdot p}{2}} dp\right].
\intertext{With $\tilde{p}_k=\left(p_1,\ldots,p_{k-1},p_{k+1},\ldots,p_d\right)^\intercal$, $\Pspace_k=\mathbb{R}$ and $\tilde{\Pspace}_k=\mathbb{R}^{d-1}$ this becomes}
	&=-\frac{1}{\fQ(q)}\frac{1}{Z}e^{-\beta V(q)}\sum_{k=1}^d\bigg[\partial_{q_k}u(q)\int_{\tilde{\Pspace}_k}e^{-\beta\frac{\tilde{p}_k\cdot\tilde{p}_k}{2}}d\tilde{p}_k\underbrace{\int_{\Pspace_k}e^{-\beta\frac{p_k^2}{2}}p_kdp_k}_{=0}\bigg].
\end{align*}
The last integral is $0$ due to symmetry, and so
$$
\pgen{1}u(q) = 0,\quad \forall u, q.
$$

2.
Having already derived (see above)
$$
\gen\big(u(q)\fcan(q,p)\big) = -\frac{1}{Z}e^{-\beta\left(\frac{p\cdot p}{2}+V(q)\right)} p\cdot \nabla_qu(q),
$$
we apply $\gen$ a second time:
\begin{align*}
\gen^{2}\big(u(q)\fcan(q,p)\big) &= \gen\left( -\frac{1}{Z}e^{-\beta\left(\frac{p\cdot p}{2}+V(q)\right)} p\cdot \nabla_qu(q)\right) \\
	&= \left(\frac{\gamma}{\beta}\Delta_p - p\cdot\nabla_q+\nabla_qV\cdot\nabla_p+\gamma p\cdot\nabla_p+d\gamma\right)\left( -\frac{1}{Z}e^{-\beta\left(\frac{p\cdot p}{2}+V(q)\right)} p\cdot \nabla_qu(q)\right)\\
	&= \frac{\gamma}{\beta}\efcan \frac{\gamma}{\beta}\Big[p\cdot\nabla_qu(q)(-\beta(d+2)+\beta^2p\cdot p)\\
	&\qquad - \big(-\beta p\cdot\nabla_qV(q) p\cdot\nabla u(q)-\beta p\cdot H_qu(q)\cdot p\big)\\[1ex]
	&\qquad + \big(\nabla_qV(q)\cdot \nabla_q u(q)-\beta(\nabla_qV(q)\cdot p)(p\cdot\nabla_q u(q))\big)\\[1ex]
	&\qquad + \gamma\big(p\cdot\nabla_q u(q)-\beta (p\cdot p)(p\cdot\nabla_q u(q))\big)\\
	&\qquad + d\gamma p\cdot \nabla_q u(q)\Big]\\[1ex]
&=\frac{1}{Z}e^{-\beta\left(\frac{p\cdot p}{2}+V(q)\right)}\big(\gamma p\cdot\nabla_qu(q)+ p\cdot H_qu(q)\cdot p - \nabla_qV(q)\cdot \nabla_qu(q)\big),
\end{align*}
with $H_qu$ the Hessian of $u$. Applying the integral in (\ref{spatialto}) and using the defintion of $\fQ$ and $Z$ gives
\begin{align*}
G_2u(q) &=\Big({\int_\Pspace e^{-\beta\frac{p\cdot p}{2}~dp}}\Big)^{-1}\int_\Pspace e^{-\beta\frac{p\cdot p}{2}}\big(\gamma p\cdot\nabla_qu(q) + \beta p\cdot H_qu(q)\cdot p - \nabla_qV(q)\cdot \nabla_qu(q)\big)~dp\\
&=\Big(\frac{\beta}{2\pi}\Big)^{\frac{d}{2}}\Big[\underbrace{\int_\Pspace\gamma e^{-\beta\frac{p\cdot p}{2}} p\cdot\nabla_qu(q)~dp}_{=0} + \underbrace{\int_\Pspace e^{-\beta\frac{p\cdot p}{2}}p\cdot H_qu(q)\cdot p~dp}_{=\frac{1}{\beta}\left(\frac{2\pi}{\beta}\right)^\frac{d}{2}\Delta_q u(q)} \\
    &\hspace{1.5cm}-\underbrace{\int_\Pspace e^{-\beta\frac{p\cdot p}{2}}\nabla_qV(q)\cdot\nabla_qu(q)~dp}_{=\left(\frac{2\pi}{\beta}\right)\nabla_q u(q) \cdot\nabla_q V(q)}\Big]\\
\intertext{The first integral is $0$ due to symmetry. Expanding the second integral, all but the ``diagonal'' summands are $0$ due to symmetry, so only $\Delta_q$ remains as integral operator. So this finally becomes}
&= \frac{1}{\beta}\Delta_q u(q) - \nabla_q u(q)\cdot\nabla_q V(q)~.
\end{align*}

3.
We apply $\gen$ a third time:
\begin{equation}
\gen^{3}\big(u(q)\fcan(q,p)\big) = \gen\Big(\gen^{2}\big(u(q)\fcan(q,p)\big)\Big),
\end{equation}
with
\begin{equation}\label{prooflabel1}
\gen^{2}\big(u(q)\fcan(q,p)\big)=\fcan(q,p)\big(\gamma p\cdot\nabla_qu(q)+ p\cdot H_qu(q)\cdot p - \nabla_qV(q)\cdot \nabla_qu(q)\big).
\end{equation}

%
%
%
%

After simplifying, with partial help of a computer algebra system, this becomes
\begin{align*}
\gen^{3}\big(u(q)\fcan(q,p)\big) &= \frac{1}{\beta}\fcan(q,p)\cdot\big[3p\cdot H_qu(q)\cdot \nabla_q V(q) + \beta p\cdot H_qV(q)\cdot \nabla_q u(q)\\
&\qquad +2\gamma \Delta_q u(q)+\beta \gamma \nabla_q u(q)\cdot \nabla_q V(q)\\[1ex]
&\qquad -\beta \gamma^2p\cdot\nabla_q u(q) - 3\beta\gamma p\cdot H_qu(q)\cdot p\\[1ex]
&\qquad -\beta p \cdot\nabla_q(p\cdot H_qu(q)\cdot p)\big].
\end{align*}

We now take the integral over $p$. First note, that
\begin{align*}
\int_\Pspace\fcan(q,p)p\cdot H_qu(q)\cdot \nabla_q V(q)~dp &= 0\\
\int_\Pspace\fcan(q,p)p\cdot H_qV(q)\cdot \nabla_q u(q)~dp &= 0\\
\int_\Pspace\fcan(q,p)p\cdot\nabla_q u(q)~dp &= 0\\
\int_\Pspace\fcan(q,p)p\cdot\nabla_q(p\cdot H_qu(q,p)\cdot p)~dp&= 0
\end{align*}
due to symmetry, therefore all that remains is
\begin{align*}
\frac{1}{\fQ(q)}&\int_\Pspace\gen^{3}\big(u(q)\fcan(q,p)\big)~dp =\\
&= \Big(\beta \int_\Pspace e^{-\beta\frac{p\cdot p}{2}}~dp\Big)^{-1}\int_\Pspace e^{-\beta\frac{p\cdot p}{2}}\big[2\gamma\Delta_qu(q)+\beta\gamma\nabla_qu(q)\cdot\nabla_qV(q) - 3\beta\gamma p\cdot H_qu(q)\cdot p\big]~dp\\
&=\frac{1}{\beta}\Big(\frac{2\pi}{\beta}\Big)^{-d/2}\Big[2\gamma\Delta_qu(q)\Big(\frac{2\pi}{\beta}\Big)^{d/2}+\beta\gamma\nabla_q u(q)\cdot\nabla_q V(q)\Big(\frac{2\pi}{\beta}\Big)^{d/2}-3\gamma\Delta_qu(q)\Big(\frac{2\pi}{\beta}\Big)^{d/2}\Big]\\
&= -\frac{\gamma}{\beta}\Delta_qu(q) + \gamma \nabla_q u\cdot\nabla_qV(q) \\
&= -\gamma\Big(\frac{1}{\beta}\Delta_qu(q)-\nabla_qu(q)\cdot\nabla_qV(q)\Big).
\end{align*}
\end{proof}

\begin{proof}[Proof of Lemma~\ref{experror}]
$E^tu$ is $N+1$ times differentiable in $t$. Thus we can apply the Taylor expansion for Banach space valued functions to $E^t$ (see \cite[Section 4.5]{Zei95}):
$$
E^tu=\sum_{n=0}^{2N}\frac{t^n}{n!}(\partial_s^nE^s\big|_{s=0})u + \Big(\int_0^1\frac{1}{(2N)!}(1-s)^{2N}\partial_s^{2N+1}E^{st}u~ds\Big)t^{2N+1}.
$$

The $n$-th derivative of $E^s$ is
$$
\partial_s^nE^s=E^s \sum_{k=0}^{\lfloor\frac{n}{2}\rfloor}\frac{n!~s^{n-2k}}{2^kk!(n-2k)!}\pgen{2}^{n-k}.
$$
For future reference, we define the operator $\displaystyle A_n:=\sum_{k=0}^{\lfloor\frac{n}{2}\rfloor}\frac{n!s^{n-2k}}{2^kk!(n-2k)!}\pgen{2}^{n-k}$.

Evaluation at $s=0$ yields 
$$
\partial_s^nE^s\big|_{s=0}=\begin{cases}0~,& n~ \text{odd} \\ \frac{n!}{2^\frac{n}{2}(\frac{n}{2})!}\pgen{2}^\frac{n}{2}~, & n~ \text{even} \end{cases},
$$
and so
$$
\sum_{n=0}^{2N}\frac{t^n}{n!}(\partial_s^nE^s\big|_{s=0})u = \sum_{n=0}^N \frac{t^{2n}}{(2n)!}\Big(\frac{(2n)!}{2^n n!} \pgen{2}^n\Big)u = \sum_{n=0}^N\frac{t^{2n}}{2^n n!}\pgen{2}^nu.
$$

Thus the remainder is only the integral:
\begin{align*}
\bigg\|E^tu - \sum_{n=0}^N \frac{\big(\frac{t^2}{2}\pgen{2}\big)^n}{n!}u\bigg\|_{2,\muQ} &= \Big\|t^{2N+1}\int_0^1\frac{1}{(2N)!}(1-s)^{2N}\partial^{2N+1}E^{st}u~ds\Big\|_{2,\muQ} \\
&\leq \frac{t^{2N+1}}{(2N)!}\sup_{s\in [0,1]}\big\|\partial^{2N+1}E^su\big\|_{2,\muQ}\\
&= \frac{t^{2N+1}}{(2N)!}\sup_{s\in [0,1]}\big\|(E^s A_{2N+1})u\big\|_{2,\muQ}\\
&\leq \frac{t^{2N+1}}{(2N)!}\sup_{s\in [0,1]}\underbrace{\big\|E^s\big\|_{2,\muQ}}_{\leq 1} \|A_{2N+1}u\|_{2,\muQ}.
\end{align*}
Due to the choice of $u$, $ \|A_{2N+1}u\|_{2,\muQ}< \infty$. This completes the proof.
\end{proof}

\section{Spectral properties of the spatial transfer operator}
\label{sec:spatial_spectrum}

To extract quantitative metastability properties of the spatial dynamics from its transfer operator, we will use the results of Huisinga~\cite{Hui01}. More precisely, we will prove that the spatial transfer operator~$S^t$ we consider satisfies the conditions under which Theorem~\ref{metastability} above is valid. Recall that we model molecular dynamics by the Langevin equation on the canonical state space with position coordinates $q$ and momenta $p$, such that the unique invariant density of the system is the canonical density $\fcan(q,p) = f_{\Qspace}(q) f_{\Pspace}(p)$ with $f_{\Qspace}(q)\propto \exp(-\beta V(q))$ and $f_{\Pspace}(p) \propto \exp(-\tfrac12 p\cdot p)$ being called the \emph{spatial} and \emph{momentum distributions}, respectively.

Opposed to the density-based statistical description of the dynamics used in the main text, we will work with a slightly greater generality here. To this end let $p_{Lan}^t$ denote the \emph{stochastic transition function} of the Langevin process, i.e.\ if $(\bm q_t,\bm p_t)$ is a Langevin process with deterministic initial condition $\bm q_0=q$ and $\bm p_0=p$ (almost surely), then
\[
p_{Lan}^t((q,p),A) = {\text{Prob}}((\bm q_t,\bm p_t)\in A),\qquad\forall\text{ measurable } A\subset \Qspace\times\Pspace .
\]
Now we can express spatial fluctuations in the canonical density even for initial distributions not absolutely continuous to~$f_{\Qspace}$, such as the Dirac measure $\delta_{q^*}$ centered in some $q^*\in \Qspace$. In accordance with~\eqref{spatialto} we have for a measurable $A\subset \Qspace$ that
\begin{equation}
p_S^t(q^*,A) := S^t \delta_{q^*}(A) = \frac{1}{f_{\Qspace}(q^*)}\int_{\Pspace} \fcan (q^*,p) p_{Lan}^t\big((q^*,p),A\times\Pspace\big)\,dp\,.
\label{eq:spatial_general}
\end{equation}
Below we will consider transition probabilities of the spatial dynamics (with lag time $t$) between two measurable subsets $A,B\subset\Qspace$ of the configuration space, supposed that the initial condition is distributed with respect to the invariant density~$f_{\Qspace}$. Let us denote these probabilities by $p_S^t(A, B)$. We have that
\[
p_S^t(A, B) = \frac{1}{\int_A\! f_{\Qspace}(q)dq}\int_A\!\! f_{\Qspace}(q) p_S^t(q,B)dq = \frac{1}{\int_A\! f_{\Qspace}(q)dq}\int_{A\times\Pspace}\!\! f_{\Qspace}(q)f_{\Pspace}(p)p_{Lan}^t\big((q,p),B\times\Pspace\big)d(q,p).
\]

\paragraph{Reversibility of the spatial dynamics}
First, we will show that $S^t$ is self-adjoint in the weighted space $\LfQ{2}(\Qspace)$, hence its spectrum is purely real. Self-adjointness of the transfer operator is equivalent with reversibility of the corresponding process: the following result is from \cite[Proposition 1.1]{Hui01}, re-stated for our purposes.
\begin{proposition}
Fix $t>0$. Let $S^t:\LfQ{2}(\Qspace)\subset \Lw{1}{\mu_{\Qspace}}{\Qspace} \to \LfQ{2}(\Qspace)$ denote the transfer operator of the spatial dynamics for lag time $t$. Let the associated (discrete time) Markov process be denoted by $\bm q_n$, $n\in\mathbb{N}$. Then $S^t$ is self-adjoint with respect to the scalar product~$\langle\cdot,\cdot\rangle_{2,\muQ}$, i.e.\ $\langle S^t u,v\rangle_{2,\muQ} = \langle u, S^t v\rangle_{2,\muQ}$ for all $u,v\in \LfQ{2}(\Qspace)$, if and only if~$\bm q_n$ is reversible.
\end{proposition}
Reversibility in this case is equivalent with $p_S^t(A,B) = p_S^t(B,A)$ for any measurable $A,B\subset\Qspace$. Indeed, one way to define the reversed process is by setting $p_{S,rev}^t(A,B):= p_S^t(B,A)$ for any measurable $A,B\subset\Qspace$. In order to show reversibility let us start with a property of the Langevin process.\footnote{The property described in Lemma~\ref{lem:Langevin_rev} is also known as \emph{extended detailed balance condition}, see \cite[Lemma~4.10]{SchSa13}.}
\begin{lemma}	\label{lem:Langevin_rev}
Let $p_{Lan,rev}^t$ denote the transition function of the reversed Langevin process, and let $A\subset \Qspace\times\Pspace$ be a measurable set which is symmetric in the momentum coordinate, i.e.
\[
A = \left\{(q,-p)\,\big\vert\, (q,p)\in A\right\}.
\]
Then $p_{Lan}^t\big((q,p),A\big) = p_{Lan,rev}^t\big((q,-p),A\big)$ for any $q\in\Qspace$, $p\in\Pspace$.
\end{lemma}
\begin{proof}
Recall the Langevin SDE \eqref{langevindynamics}:
\begin{eqnarray*}
\partial \bm q_t & = & \bm p_t\\
\partial \bm p_t & = & -\nabla V(\bm q_t) - \gamma \bm p_t + \sigma \bm w_t\,.
\end{eqnarray*}
The reversed Langevin process is also an It\^{o} diffusion~\cite{HuPa86} governed by the SDE
\begin{eqnarray*}
\partial \bm q_t & = & -\bm p_t\\
\partial \bm p_t & = & \nabla V(\bm q_t) - \gamma \bm p_t + \sigma \bm w_t\,.
\end{eqnarray*}
Applying the substitution $\tilde{p} = -p$ for the Langevin equations in forward time, and using the fact that $\bm w_t$ and $-\bm w_t$ are stochastically equivalent in the sense that their distributions coincide, we obtain
\begin{eqnarray*}
\partial \bm q_t & = & -\tilde{\bm p}_t\\
\partial \tilde{\bm p}_t & = & \nabla V(\bm q_t) - \gamma \tilde{\bm p}_t + \sigma \bm w_t\,.
\end{eqnarray*}
Note that this is the same SDE as for the reversed process. Thus, the reversed process starting at $(q,-p)$ has the same distribution as $(\bm q_t, -\bm p_t)$, where $(\bm q_t,\bm p_t)$ is the forward time process starting at $(q,p)$.
\end{proof}
To show reversibility of the spatial dynamics, we rewrite $p_S^t(A,B)$ in equivalent terms. Let us also introduce the shorthand notation $F_A = \muQ(A) = \int_A f_{\Qspace}(q)\,dq$.
\begin{eqnarray*}
p_S^t(A,B) & = & F_A^{-1}\int_{A\times\Pspace} f_{\Qspace}(q)f_{\Pspace}(p) p_{Lan}^t\big((q,p),B\times\Pspace\big)\,d(q,p) \\
	& = & (-1)^d F_A^{-1} \int_{A\times-\Pspace} f_{\Qspace}(q)f_{\Pspace}(-\tilde p) p_{Lan}^t\big((q,-\tilde p\big),B\times\Pspace)\,d(q,\tilde p)\\
	& = & F_A^{-1}\int_{A\times \Pspace} f_{\Qspace}(q)f_{\Pspace}(-\tilde p) p_{Lan}^t\big((q,-\tilde p),B\times\Pspace\big)\,d(q,\tilde p)\,,
\end{eqnarray*}
where we used the integral substitution $\tilde p = -p$, then the symmetry of~$\Pspace$, such that flipping the integration bounds only introduces change of sign. From this and Lemma~\ref{lem:Langevin_rev} we obtain
\begin{equation}
p_S^t(A,B) = F_A^{-1}\int_{A\times\Pspace} f_{\Qspace}(q)f_{\Pspace}(\tilde p) p_{Lan,rev}^t\big((q,\tilde p),B\times\Pspace\big)\,d(q,\tilde p)\,,
\label{eq:p_sAB}
\end{equation}
by exploiting that $f_{\Pspace}(-\tilde p) = f_{\Pspace}(\tilde p)$. In the next lemma we establish that the right hand side of \eqref{eq:p_sAB} in fact expresses the transition probability from $A$ to $B$ for the \emph{reversed spatial process} $p_{S,rev}^t$, and hence $p_S^t(A,B) = p_{S,rev}^t(A,B) = p_S^t(B,A)$. This concludes the proof of reversibility for the spatial process.

\begin{lemma}
It holds $p_{S,rev}^t(A,B) = F_A^{-1}\int_{A\times\Pspace} f_{\Qspace}(q)f_{\Pspace}(p) p_{Lan,rev}^t\big((q, p),B\times\Pspace\big)\,d(q,p)$ for any measurable $A,B\subset\Qspace$.
\end{lemma}
\begin{proof}
The transition probability for the reversed system can be obtained from Bayes formula:
\begin{eqnarray*}
p_{S,rev}^t(A,B) & = & \frac{p_S^t(B,A)\int_B f_{\Qspace}(q)dq}{\int_A f_{\Qspace}(q)dq}\\
	& = & F_A^{-1} \int_B f_{\Qspace}(q)p_S^t(q,A)\,dq\\
	& = & F_A^{-1} \int_{q\in B}\int_{\tilde{q}\in A} f_{\Qspace}(q)p_S^t(q,d\tilde{q})\,dq\\
	& = & \ldots
\end{eqnarray*}
where, by writing $\int_{q\in B}$ we would like to indicate which variable is integrated over which set, to maintain a good readability. With~\eqref{eq:spatial_general} we can expand the term on the right hand side further:
\begin{eqnarray*}
\ldots & = &  F_A^{-1}\int_{q\in B} \int_{p\in\Pspace}  \int_{\tilde{q}\in A}\int_{\tilde{p}\in\Pspace} \underbrace{f_{\Qspace}(q) f_{\Pspace}(p) p_{Lan}^t\big((q,p),d(\tilde{q},\tilde{p})\big)\,dp dq}_{\stackrel{\rm def}{=} f_{\Qspace}(\tilde{q})f_{\Pspace}(\tilde{p})p_{Lan,rev}^t\big((\tilde{q},\tilde{p}),d(q,p)\big)d\tilde{q}d\tilde{p}}\\
	& = & \ldots
\end{eqnarray*}
where, for the underbraced term, we use the infinitesimal version of Bayes formula to relate the transition functions of the forward and backward time Langevin processes. Rearranging the integration order yields the claim
\[
\ldots =  F_A^{-1} \int_{\tilde{q}\in A} \int_{\tilde{p}\in\Pspace} f_{\Qspace}(\tilde{q})f_{\Pspace}(\tilde{p}) p_{Lan,rev}^t\big((\tilde{q},\tilde{p}),B\times\Pspace\big) d\tilde{p}d \tilde{q}\,.
\]
\end{proof}

\paragraph{Geometric ergodicity of the spatial dynamics}
\begin{definition}
Let $\bm x_t$, where $t$ denotes either discrete or continuous time, be a Markov process with transition function $p^t$ and unique invariant measure~$\mu$. Then $\bm x_t$ is called \emph{geometrically ergodic} if for every state $x\in \mathcal{X}$ and time~$t$
\[
\|p^t(x,\cdot) - \mu\|_{\rm TV} \le M(x)\rho^t
\]
holds for some $M\in L^1_{\mu}$ and $\rho<1$. Here, $\|\cdot\|_{\rm TV}$ denotes the total variation norm for signed measures.
\end{definition}
Following \cite[Proposition~6.3]{Hui01} (see also \cite{MaSt02} and \cite{MaStHi02}) we can establish the geometric ergodicity of the Langevin process. Here and in the following~$\mucan$ denotes the canonical measure of the Langevin process, i.e.\ $d\mucan = \fcan dm$.
\begin{proposition}	\label{prop:Langevin_geomerg}
Let $\bm x_t$ denote the Langevin process. Fix some lag time $t>0$, and let $\bm z_n:= \bm x_{nt}$ be the sampled time process. In either of the following cases,~$\bm z_n$ has the unique invariant measure $\mucan$ and is geometrically ergodic.
\begin{enumerate}[(i)]
\item The state space $\mathcal{X}\subset\mathbb{R}^d$ is periodic and the potential $V:\mathcal{X}\to\mathbb{R}$ is smooth.
\item The state space $\mathcal{X} = \mathbb{R}^d$, the potential $V:\mathcal{X}\to\mathbb{R}_{\ge 0}$ is smooth and $V(x)$ is growing at infinity as $\|x\|^{2l}$ for some positive integer~$l$.
\end{enumerate}
\end{proposition}

We assume from now on that either condition (i) or (ii) of Proposition~\ref{prop:Langevin_geomerg} is satisfied. This means, in particular, that for a fixed $t>0$ there is a function $M\in \Lw{1}{\mucan}{\Qspace\times\Pspace}$ and a constant $\rho<1$ such that for every $q\in\Qspace$ and $p\in\Pspace$ holds
\begin{equation}
\|p_{Lan}^{nt}((q,p),\cdot) - \mucan \|_{\rm TV} \le M(q,p)\rho^n\,.
\label{eq:Langevin_geomerg}
\end{equation}
To show geometric ergodicty of the spatial dynamics, recall its transition function from~\eqref{eq:spatial_general}, and that~$\muQ$ given by $d\muQ(q):= \fQ(q)dq$ is its unique invariant measure. Note that by construction $\muQ = \mucan (\cdot\times\Pspace)$. We have
\begin{eqnarray*}
\big\|p_S^{nt}((q,p),\cdot) - \muQ\big\|_{\rm TV} & = & \Big\|\int_{\Pspace}\tfrac{\fcan(q,p)}{\fQ(q)}p_{Lan}^{nt}\big((q,p),\cdot\times\Pspace\big)\,dp - \mucan(\cdot\times\Pspace)\Big\|_{\rm TV}\\
	& = & \Big\|\int_{\Pspace}\tfrac{\fcan(q,p)}{\fQ(q)}\left[p_{Lan}^{nt}\big((q,p),\cdot\times\Pspace\big)- \mucan(\cdot\times\Pspace)\right]\,dp \Big\|_{\rm TV}\\
	& \le & \int_{\Pspace}\tfrac{\fcan(q,p)}{\fQ(q)}\big\|p_{Lan}^{nt}\big((q,p),\cdot\times\Pspace\big)- \mucan(\cdot\times\Pspace)\big\|_{\rm TV}\,dp \\
	& \le & \rho^n\underbrace{\int_{\Pspace}\tfrac{\fcan(q,p)}{\fQ(q)}M(q,p)\,dp}_{=: \tilde{M}(q)}\,,
\end{eqnarray*}
where the second line follows from $\int_{\Pspace}\tfrac{\fcan(q,p)}{\fQ(q)}dp=1$, the third comes from pulling the norm inside the integral, while the last inequality is a consequence of~\eqref{eq:Langevin_geomerg}, and of the total variation norm being defined by a supremum over all possible partitions of state space (more precisely: restricting the partitions to sets of the form ${\cdot\times\Pspace}$ results in a total variation not greater than in the non-restricted case). $M\in \Lw{1}{\mucan}{\Qspace\times\Pspace}$ implies $\tilde{M}\in \Lw{1}{\muQ}{\Qspace}$. Thus, the spatial process is geometrically ergodic.

\paragraph{The spectrum of the spatial transfer operator}
\emph{Dynamical properties} of the transition function $p_S^t$---namely reversibility and geometric ergodicity---imply some desirable \emph{spectral properties} of the associated transfer operator\footnote{As shown by Baxter and Rosenthal \cite[Corollary to Lemma~1]{BaRo95}, the transfer operator associated with a transition function having the invariant measure~$\mu$ is a well-defined contraction on every $\Lw{r}{\mu}{\Qspace}$, $1\le r\le \infty$. See also Corollary~\ref{cor:opcontraction} in this manuscript.\label{ftn:BaRo95}} $S^t:\LfQ{2}(\Qspace) \to \LfQ{2}(\Qspace)$.

Consider the two properties of some transfer operator $P$:
\begin{enumerate}
\item[(P1)] The essential spectral radius of $P$ is less than one, i.e.\ $r_{\rm ess}(P)<1$.
\item[(P2)] The eigenvalue $\lambda=1$ of $P$ is simple and dominant.
\end{enumerate}

\cite[Theorem~4.31]{Hui01} states:
\begin{theorem}
Let $P:\mathcal{L}^2_{\mu}\to \mathcal{L}^2_{\mu}$ be a transfer operator associated with the reversible stochastic transition function $p$. Then $P$ satisfies properties (P1) and (P2) in~$\mathcal{L}^2_{\mu}$, if and only if~$p$ is $\mu$-irreducible and ($\mu$-a.e.) geometrically ergodic. The latter two conditions on~$p$ are satisfied, in particular, if~$p$ is geometrically ergodic.
\end{theorem}

Thus, we have shown
\begin{corollary}
If the potential $V$ satisfies either conditions in Proposition~\ref{prop:Langevin_geomerg}, then the spatial transfer operator $S^t:\LfQ{2}(\Qspace) \to \LfQ{2}(\Qspace)$ is self-adjoint and has the properties (P1) and (P2). These are exactly the conditions under which Theorem~\ref{metastability} holds.
\end{corollary}

\section{Smoothness of eigenfunctions}	\label{sec:smoothness}

We are going to show in this section that if Assumption~\ref{assu:smoothness} holds for the potential, then the eigenfunctions of the
\begin{enumerate}[(i)]
\item spatial transfer operator $S^t$ associated with the Langevin process~\eqref{langevindynamics}, and of the
\item  generator~$G$ of the Smoluchowski process~\eqref{smoluchowskidynamics}
\end{enumerate}
are smooth, i.e.\ $\mathcal{C}^{\infty}$ functions. Qualitative results of this kind go back to H\"ormander~\cite{Hor67}, where hypoelliptic diffusions have been considered. Indeed, what is known as \emph{H\"ormander's condition}, is equivalent to $W_{\ell}(x)>0$ for some $\ell\ge 1$, the function $W_{\ell}$ being defined below. Since the spatial transfer operator involves an averaging over the momenta, we will require quantitative estimates on the smoothness (cf.~\eqref{eq:smoothbound}) of transition density functions in order to carry over the smoothness to the spatial transition density function, and to the  eigenfunctions of~$S^t$.
For this we will use the results based on Malliavin calculus.

\paragraph{The general theory}
Consider the stochastic differential equation
\begin{equation}
\partial_t \bm x_t = b_0(\bm x_t) + \sum_{i=1}^m b_i(\bm x_t)\bm w_t^i,
\label{eq:SDEgeneral}
\end{equation}
where the $b_0,b_1,\ldots,b_m$ are vector fields over $\mathbb{R}^d$, and the $\bm w_t^i$ are independent standard one-dimensional ``white noise'' processes. Note that both the Langevin and the Smoluchowski differential equations we consider here are special cases of this, e.g.\ the latter with~$m=d$ and $b_i\equiv\textrm{const}\cdot e_i$, $i=1,2,\ldots,d$, with a suitable constant, and $e_i\in\mathbb{R}^d$ being the canonical unit vectors.

For two differentiable vector fields $v_1, v_2:\mathbb{R}^d\to\mathbb{R}^d$ define their \emph{Lie bracket} by $[v_1,v_2](x) = Dv_2(x)v_1(x) - Dv_1(x)v_2(x)$, where $Dv$ is the derivative of $v$ with entries $(Dv(x))_{ij} = \tfrac{\partial v_i}{\partial x_j}(x)$. Further, for a multi-index $\alpha = (\alpha_1,\ldots,\alpha_k)\in \{0,1,\ldots,m\}^k\cup\{\emptyset\}$ define $b_i^{\alpha}$, $1\le i\le m$, by induction: $b_i^{\emptyset} = b_i$, $1\le i\le m$, and $b_i^{(\alpha,j)} = [b_j,b_i^{\alpha}]$, $0\le j\le m$, where $(\alpha,j) = (\alpha_1,\ldots,\alpha_k,j)$. Also, let $(\emptyset,j)=j$. For a multi-index $\alpha$ define
\begin{align*}
|\alpha| &= \left\{\begin{array}{ll}
0, & \text{if }\alpha=\emptyset\\
\ell, & \text{if }\alpha\in\{0,\ldots,m\}^{\ell}
\end{array}\right.\\
\|\alpha\| &= \left\{\begin{array}{ll}
0 ,& \text{if }\alpha=\emptyset\\
|\alpha| + \textrm{card}\{j\,\vert\,\alpha_j=0\}, & \text{if }|\alpha|\ge 1
\end{array}\right.
\end{align*}
Now, for some $\ell\in\mathbb{N}$ and $x,\eta\in\mathbb{R}^d$ let
\begin{align*}
W_{\ell}(x,\eta) &= \sum_{i=1}^m\sum_{\|\alpha\|\le \ell-1} (\eta\cdot b_i^{\alpha}(x))^2,\\
W_{\ell}(x) &= 1 \wedge \inf_{\|\eta\|_2=1}W_{\ell}(x,\eta) := \min\{1,\inf_{\|\eta\|_2=1}W_{\ell}(x,\eta)\},
\end{align*}
where the dot denotes the usual scalar product and $\|\cdot\|_2$ the Euclidean norm. Kusuoka and Stroock prove the following result; see also in~\cite{BaTa95}.
\begin{theorem}[{\cite[Corollary 3.25]{KuSt85}}]	\label{thm:smoothness}
Let $b_0,\ldots,b_m$ be $\mathcal{C}^{\infty}$ with all derivatives of order greater or equal one be bounded, and fix some $\ell\in\mathbb{N}$. Then for every $x\in A_{\ell}:=\{x\in\mathbb{R}^d\,\vert\,W_{\ell}(x)>0\}$, $t>0$, the stochastic transition function $p^t$ of~\eqref{eq:SDEgeneral} has a smooth (i.e.\ $\mathcal{C}^{\infty}$ in every variable) transition density function $k$ with respect to the Lebesgue measure, i.e.\ $p^t(x,dy) = k(t,x,y)dy$.\\
Moreover, for $n,n_x,n_y\in\mathbb{N}_{\ge0}$, multi-indices $\alpha\in\{0,1,\ldots,n_x\}^d$, $\beta\in\{0,1,\ldots,n_y\}^d$, and some $T>0$, there are constants $c_1, c_2, r>0$, independent of $x$ and $y$, such that
\begin{equation}
\big|\partial_t^n\partial_x^{\alpha}\partial_y^{\beta}k(t,x,y)\big| \le \frac{c_1}{t^rW_{\ell}(x)}\exp\left(-c_2\frac{\|x-y\|_2^2}{t}\right)\quad\forall x\in A_{\ell},\ y\in\mathbb{R}^d,\ 0<t\le T.
\label{eq:smoothbound}
\end{equation}
\end{theorem}

\paragraph{Smoothness of the Smoluchowski eigenfunctions}
Consider the Smoluchowski equation
\[
\partial_t \bm q_t = -\nabla V(\bm q_t) + \sigma\bm w_t,
\]
with $\sigma>0$. Note that if~$V$ satisfies the conditions of Assumption~\ref{assu:smoothness}, then the Smoluchowski SDE automatically satisfies the first condition of Theorem~\ref{thm:smoothness}.

We have that $b_i^{\emptyset}\equiv \sigma e_i$, the~$i^{\rm th}$ column of the~$d\times d$ matrix~$\sigma I$ (here,~$I$ denotes the identity matrix). Hence, with $B = \sigma I$,
\begin{equation}
W_1(q) = 1\wedge \inf_{\|\eta\|_2=1}\sum_{i=1}^d (\eta\cdot\sigma e_i)^2 = 1\wedge \inf_{\|\eta\|_2=1}\|\eta^T B\|_2^2 = 1\wedge \sigma^2>0,\quad\forall q\in\mathbb{R}^d\,.
\label{eq:Wbound}
\end{equation}
In particular, by Theorem~\ref{thm:smoothness}, for all~$q\in\mathbb{R}^d$ the stochastic transition function of the above Smoluchowski equation has a smooth transition density function $k(t,q,\hat q)$ with
\begin{equation}
|\partial_y^{\alpha}k(t,q,\hat q)| \le \tilde{c}_1 \exp(-\tilde{c}_2\|q-\hat q\|_2^2),\quad \forall q,\hat q\in\mathbb{R}^d,
\label{eq:derivativebound}
\end{equation}
for a fixed $t>0$, with some constants $\tilde{c}_1, \tilde{c}_2$ independent of $q$ and $\hat q$.

Let $G$ and $\ops{t}$ denote the generator and the transfer operator (with respect to the Lebesgue measure) of the Smoluchowski process, respectively. Now, by the Spectral Mapping Theorem (Theorem~\ref{spectralmappingtheorem}) $G u = \lambda u$ if and only if $\ops{t} u = e^{\lambda t}u$ for any $t\ge 0$, i.e.\
\[
e^{\lambda t}u(\hat q) = \int_{\mathbb{R}^d} u(q) k(t,q,\hat q)\,dq.
\] 
By Lebesgue's theorem on differentiating parameter-dependent integrals, the right hand side of this equation is continuously differentiable (with respect to~$\hat q$) because for any $i\in\{1,\ldots,d\}$,
\[
|u(q) \partial_{\hat{q}_i}k(t,q,\hat q)| \le |u(q)|\,\tilde{c}_1 \exp\left(-\tilde{c}_2\|q-\hat q\|_2^2\right) \le \tilde{c}_1|u(q)|,
\]
which is an integrable function. This argument can be iterated for derivatives of any order, showing that $u(q)$, the eigenfunction of $G$, is smooth.

\begin{remark}	\label{rem:measurechange}
We have also considered the generator and semigroup with respect to the canonical measure with density $f_{\Qspace}(q) \propto \exp(-\beta V(q))$. As noted in the discussion preceding Corollary~\ref{g2smolu}, the densities of the transfer operator with respect to the Lebesgue and the canonical measure differ only up to a factor $f_{\Qspace}(q)$, which is a nowhere zero smooth function. Thus the Smoluchowski eigenfunctions are smooth, no matter with respect to which of these both measures we consider them.
\end{remark}

\paragraph{Smoothness of the spatial eigenfunctions}
Consider the Langevin equation, as in \eqref{langevindynamics},
\[
\begin{aligned}
\partial_t\bm{q}_t &= \bm{p}_t\\
\partial_t\bm{p}_t &= -\nabla V(\bm{q}_t) - \gamma \bm{p}_t +\sigma \bm{w}_t~,
\end{aligned}
\]
with $\gamma,\sigma>0$.

In order to show the smoothness of the (time-dependent) spatial eigenfunctions satisfying $S^t u^t = \lambda_t u^t$, our strategy is the same as for the Smoluchowski case. First, we show that $W_{\ell}$ is uniformly bounded away from zero as in~\eqref{eq:Wbound}, then apply Theorem~\ref{thm:smoothness} to imply a bound on the derivatives of the transition density function $k$, as in~\eqref{eq:derivativebound}. From this, using Lebesgue's theorem, we show the smoothness of the right hand side of the eigenvalue equation
\[
\lambda_t u(\hat q) = S^t u^t(\hat q) = \frac{1}{f_{\Qspace}(q)}\int_{\mathbb{R}^d}\iint_{\mathbb{R}^{2d}} u^t(q)\fcan(q,p)k\big(t,\left(q,p\right),\left(\hat q,\hat p\right)\big)\,d(q,p)\,d\hat p\,.
\]
Apart from bounding $W_{\ell}$, every step is analogous as in the Smoluchowski case, hence we will omit them.

As for $W_{\ell}$, note that for the Langevin equation
\[
b_0(q,p) = \begin{pmatrix}
p \\ -\gamma p - \nabla V(q)
\end{pmatrix},
\qquad b_i(q,p) = \begin{pmatrix}
0\\ \sigma e_i.
\end{pmatrix},\ i\in\{1,\ldots,d\}.
\]
It follows that (see also~\cite[Theorem~3.2]{MaStHi02})
\[
b_i^{\emptyset}\equiv\begin{pmatrix}
0\\ \sigma e_i.
\end{pmatrix},\qquad b_i^{0}\equiv \begin{pmatrix}
\sigma e_i \\ -\gamma \sigma e_i
\end{pmatrix},\ i\in\{1,\ldots,d\}.
\]
Note that these latter $2d$ vectors span $\mathbb{R}^{2d}$, hence the $2d\times 2d$ matrix~$B$ with columns~$b_i^{\emptyset}$ and~$b_i^0$ is non-singular. Recalling the definition of~$W_{\ell}$, we have that
\[
W_3\big((q,p)\big) \ge 1\wedge \inf_{\|\eta\|_2=1}\|\eta^T B\|_2^2 > s_{\min}^2\quad\forall q,p\in\mathbb{R}^d,
\]
where $s_{\min}>0$ is the smallest singular value of~$B$. Thus, $W_3$ is uniformly bounded away from zero, concluding the proof of smoothness of~$u^t$ for any~$t>0$:
\begin{proposition}	\label{prop:smooth_spatial}
If the potential satisfies Assumption~\ref{assu:smoothness}, then the eigenfunctions of the spatial transfer operator~$S^t$ are smooth for every~$t>0$.
\end{proposition}

\paragraph{From $\mathbb{R}^d$ to periodic systems}
Since in molecular dynamics often periodic angular coordinates are used, it is interesting to see whether the smoothness results from above apply also for periodic systems on a torus.

To this end, let $X = \mathbb{T}^d$ be the unit torus, which we identify with $[0,1)^d\subset\mathbb{R}^d$. We consider an It\^{o} diffusion of the form~\eqref{itodiffusion} on~$X$. Let~$b:\mathbb{T}^d\to\mathbb{R}^d$ and~$\Sigma:\mathbb{T}^d\to\mathbb{R}^{d\times d}$ be smooth functions.

Set $\widehat{b}$ and $\widehat{\Sigma}$ to be the periodic extensions of $b$ and $\Sigma$ to $\mathbb{R}^d$, respectively, i.e.\ $\widehat{b}(x+r) = b(x)$ for all $x\in [0,1)^d$ and $r\in \mathbb{Z}^d$, and analogously for $\Sigma$. Assume that $b$ and $\Sigma$ are such that $\widehat{b}$ and $\widehat{\Sigma}$ satisfy all conditions of Theorem~\ref{thm:smoothness}, and that there is an~$\ell\ge 1$ such that~$W_{\ell}$ is uniformly bounded away from zero (as it is the case for the Smoluchowski and Langevin dynamics). Then the associated SDE has a smooth transition density function~$\widehat{k}$ such that its derivatives of any fixed order show an exponential decay:
\[
\left|\partial_y^{\alpha}\widehat{k}(t,x,y+r)\right|  = \mathcal{O}\big(\exp(-c\|r\|_2^2)\big)\ \text{as } \|r\|_2\to\infty\,,
\]
for $x,y\in [0,1)^d$ and some $c>0$ independent of $x,y$. An analogous bound holds for derivatives with respect to~$t$ as well.

Note that for a fixed $x$, the function $(t,y)\mapsto \widehat{k}(t,x,y)$ solves the Fokker--Planck equation~\eqref{fokkerplanckequationito} on $\mathbb{R}_+\times\mathbb{R}^d$. Thus, since this partial differential equation is linear, and $\widehat{b}$ and $\widehat{\Sigma}$ are periodic, the transition density function
\begin{equation}
k(t,x,y) := \sum_{r\in\mathbb{Z}^d}\widehat{k}(t,x,y+r)
\label{eq:collapseddensity}
\end{equation}
formally solves the Fokker--Plack equation on~$\mathbb{T}^d$ associated with~$b$ and~$\Sigma$. For this statement to be rigorous, we have to show that~$k(t,x,y)$, defined as here, is continuously differentiable in~$t$ and twice continuously differentiable in~$y$. This is achieved by showing the uniform convergence of the summand-wise differentiated sum in~\eqref{eq:collapseddensity}. Setting $\rho=\|r\|_2$, observe that there are $\mathcal{O}(\rho^{d-1})$ many sets of the form $[0,1)^d+r$, $r\in\mathbb{Z}^d$, intersected by a sphere of radius~$\rho$. Hence, we can estimate
\[
\sum_{r\in\mathbb{Z}^d}\left|\partial_y^{\alpha}\widehat{k}(t,x,y+r)\right| \le C \sum_{r\in\mathbb{Z}^d} \exp\left(-c\|r\|_2^2\right) \le \tilde{C} \int_0^{\infty}\rho^{d-1}e^{-c\rho^2}\,d\rho < \infty\,,
\]
with some constants $c,C,\tilde{C}>0$, independent of~$y$. This shows uniform convergence of the sum, and the smoothness of~$k(t,x,y)$ in~$y$. An analogous computation can be done for~$t$. Thus,~$k$ is smooth and solves the Fokker--Plack equation on~$\mathbb{T}^d$. Since $\widehat{k}(t,x,\cdot)$ converges (weakly) to the Dirac distribution centered in~$x$ as~$t\to 0$, so does~$k(t,x,\cdot)$. These last two properties show that~$k$ is the (unique) transition density function associated with the It\^{o} diffusion on the torus; and we have shown that it is smooth.

Note that if the potential~$V:\mathbb{T}^d\to\mathbb{R}$ is smooth, then $\widehat{V}(x+r):= V(x)$, $r\in\mathbb{Z}^d$, readily satisfies Assumption~\ref{assu:smoothness}. Thus we have
\begin{proposition}	\label{prop:smooth_periodic_spatial}
Let the potential $V:\mathbb{T}^d\to\mathbb{R}$ be smooth. Then the eigenfunctions of the associated spatial transfer operator~$S^t$ are smooth for every~$t>0$.
\end{proposition}

\bibliographystyle{siam}
\bibliography{Bibliography_googlescholar}

\end{document}